\documentclass[preprint,12pt]{elsarticle}

\usepackage{amsmath,amssymb}
\usepackage{graphicx}
\usepackage{microtype}
\usepackage{multirow}
\usepackage{pifont}
\usepackage{xcolor}
\usepackage{amsthm}
\usepackage{booktabs}
\usepackage{xurl}
\usepackage{comment}
\usepackage{algorithm}
\usepackage{algpseudocode}
\usepackage{etoolbox}
\algrenewcommand\algorithmicrequire{\textbf{Input:}}
\newcommand{\bluecheck}{{\color{teal}\checkmark}}
\newcommand{\xmark}{\color{red}\ding{55}}
\newtheorem{Theorem}{Theorem}

\newtheorem*{AppendixTheorem}{Theorem}
\newtheorem*{AppendixLemma}{Lemma}
\newtheorem{remark}{Remark}

\journal{Journal of Computational and Applied Mathematics}

\begin{document}

\begin{frontmatter}

\title{Deep Learning-based Surrogate Modelling of the LOD Method for Multiscale Problems}

\author[cau]{Marc Haltmayer\fnref{fn1}}
%\ead{haltmayermarc@cau.ac.kr}

\author[cau]{Jaemin Seo\fnref{fn1}}
%\ead{tjwoals217@cau.ac.kr}

\author[cau]{Yuseung Lee}
%\ead{idia0780@cau.ac.kr}

\author[samsung]{Sungyeop Lee}
%\ead{sungyeop.lee@samsung.com}

\author[samsung]{Jaehoon Jeong}
%\ead{jh77.jeong@samsung.com}

\author[cau]{Jae Yong Lee\corref{cor1}}
%\ead{jaeyong@cau.ac.kr}

\cortext[cor1]{Corresponding author.}
\fntext[fn1]{Equal contribution.}

\affiliation[cau]{organization={Department of Artificial Intelligence, Chung-Ang University},
            addressline={84 Heukseok-ro, Dongjak-gu},
            city={Seoul},
            postcode={06974},
            country={Republic of Korea}}

\affiliation[samsung]{organization={Computational Science and Engineering Team, Samsung Electronics},
            addressline={41 Samsung-ro, Giheung-gu},
            city={Yongin-si},
            postcode={17113},
            country={Republic of Korea}}

\begin{abstract}
Multiscale problems are notoriously difficult to tackle using traditional numerical methods, as accurately resolving fine-scale features often requires prohibitively fine discretizations. This challenge is particularly pronounced in applications such as materials science, fluid dynamics, climate systems, chemical processes, and complex networks. Recent neural operator models provide a promising data-driven alternative, but frequently struggle to achieve sufficient accuracy in the presence of strongly heterogeneous or oscillatory coefficients. In this work, we focus on the solution of elliptic PDEs with rough and high-contrast inputs. The Localized Orthogonal Decomposition (LOD) method is a well-established numerical approach for such problems, but it comes, however, at a substantial computational cost. We investigate the performance of popular neural operator architectures on these challenging multiscale problems and identify key limitations in their ability to resolve fine-scale structure. To overcome these challenges, we introduce \textbf{LOD-MSNO} (LOD-Multiscale Neural Operator), a hybrid approach that leverages the LOD method as a strong multiscale prior by building on its representation of the solution as a linear combination of problem-adapted basis functions, while addressing its main computational bottlenecks through data-driven operator learning. We further provide theoretical error estimates for the proposed coefficient-learning framework. Lastly, we demonstrate the potential of our proposed method to outperform current neural operator baselines in terms of accuracy for challenging multiscale inputs, while mainly retaining the computational efficiency of neural operator models.
\end{abstract}

\begin{keyword}
Scientific Machine Learning \sep Multiscale Problems \sep Neural Operator \sep Finite Element Method  
\end{keyword}

\end{frontmatter}

\section{Introduction}

Multiscale PDE problems, characterized by the coexistence of important features across multiple spatial and temporal scales, arise in a wide range of scientific and engineering domains. A representative example is given by Darcy Flow,
\begin{equation}
\label{darcy_eq}
\begin{aligned}
- \nabla \cdot (\kappa \nabla u) &= f \quad &&\text{in } D, \\
u &= 0 \quad &&\text{on } \partial D,
\end{aligned}
\end{equation}
where $D \subset \mathbb{R}^d$, $d = 2,3$ is a bounded physical domain, $u : D \to \mathbb{R}$ denotes the pressure, $f \in L^2(D)$ is a given source term and $\kappa : D \to \mathbb{R}_{+}$ is a heterogeneous permeability coefficient. The Darcy Flow equation can be viewed as a heterogeneous extension of Poisson's equation, which arises naturally across a range of physical settings involving transport through complex media. For example, it describes fluid flow in porous subsurface \citep{bear1988dynamics} formations, while analogous formulations govern heat conduction in materials \citep{Yang2024} with spatially varying thermal properties and potential fields in electrically inhomogeneous media \citep{KIRKHAM2023431}. Across these problems, the spatial variability of the coefficient field plays a central role, as even fine-scale heterogeneities can significantly influence the global solution behavior, leading to pronounced multiscale effects. This interplay between fine-scale structure and macroscopic response poses significant challenges for both modeling and computation, often requiring high-resolution simulations to accurately capture the underlying physics. Such demands are particularly critical in practical engineering workflows, where repeated evaluations under varying material configurations or design parameters are often required. As a result, developing accurate and efficient solution strategies for this class of problems remains a key objective in both scientific computing and industrial applications~\cite{efendiev2009multiscale,hou1997multiscale}. 

\begin{figure*}[t]
  \centering
    \includegraphics[width=0.99\linewidth]{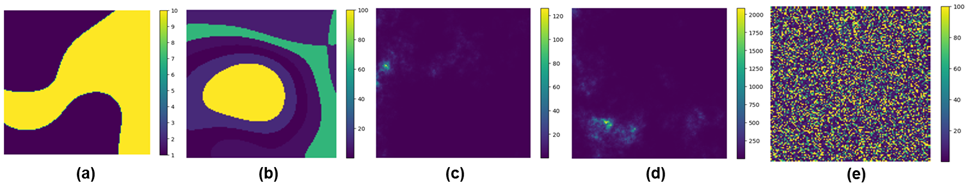}
    \caption{Example permeability fields $\kappa(x)$. (a) Two-phase piecewise-constant as considered in many related works. (b) Six-phase piecewise-constant. (c) lognormal random field. (d) High-contrast lognormal random field spanning several orders of magnitude, yielding a particularly difficult regime due to extreme coefficient variation.  (e) Cell-wise i.i.d.\ discrete random field sampled from a finite level set.}
  \label{fig:Darcy_coeffs}
\end{figure*} 
%Explicitly resolving all fine-scale features of $\kappa$ on a sufficiently fine mesh leads to large-scale discretizations with millions or billions of degrees of freedom, rendering direct numerical simulation prohibitively expensive in terms of memory and computational cost. This challenge is further exacerbated when the problem must be solved repeatedly for different such permeability fields $\kappa$. Consequently, there is a growing need to surrogate such numerical simulations using neural networks in order to overcome these computational limitations. 
In this context, neural operator methods have emerged as promising alternatives, as they aim to learn mappings between infinite-dimensional function spaces and thereby approximate solution operators of parametric PDEs directly. However, despite their favorable computational properties and remarkable success in learning solution operators of parametric PDEs, neural operator models exhibit important limitations when applied to multiscale problems with rough and strongly heterogeneous inputs. Much of the existing neural operator literature focuses on relatively simplified flow configurations; for instance, benchmark studies involving Fourier Neural Operator (FNO) \citep{li2021fourierneuraloperatorparametric}, Deep Operator Network (DeepONet) \citep{fair}, and Transformer-based architectures such as Transolver \citep{Transolver}, and related approaches frequently consider two-phase Darcy flow problems with moderately varying permeability fields (Figure~\ref{fig:Darcy_coeffs}(a)). While these settings provide useful testbeds for model development and evaluation for such models, they do not fully reflect the complexity encountered in realistic subsurface flow applications. In particular, multi-phase Darcy flow in highly heterogeneous and highly oscillatory media (Figure~\ref{fig:Darcy_coeffs}(b)-(e)) remains largely unexplored in the neural operator community, despite its relevance in groundwater hydrology, petroleum engineering, and carbon sequestration.

% \textbf{Contributions.}
\paragraph{Contributions}
In this work, we address these gaps and make the following contributions.

\begin{itemize}
    \item We identify a key shortcoming of existing neural operator models when applied to multiscale problems with rough, high-contrast coefficients. Using Darcy flow as a prototypical example, we demonstrate the challenges posed by highly heterogeneous and multiscale permeability fields through a series of carefully designed datasets inspired by applications in industry that exhibit complex fine-scale features.
    \item We propose \textbf{LOD-MSNO}, an operator learning model specifically designed for multiscale PDE problems that follows a coefficient-learning paradigm. The approach adopts a hybrid strategy by using the finite element method-based Localized Orthogonal Decomposition (LOD) approximation as a strong numerical prior and employing neural networks to learn the associated multiscale basis functions and expansion coefficients.
    \item We provide a theoretical analysis showing that the proposed model converges to the true numerical LOD solution. As our main theorem, we derive a precise error bound with respect to the true analytical solution, separating the LOD discretization error from the generalization error of the neural surrogate.
    \item We demonstrate that our method can outperform established neural operator benchmarks on the proposed multiscale datasets. These results validate the feasibility of our idea and highlight its potential for reliable and efficient surrogate modeling of multiscale PDE problems.
\end{itemize}

\section{Related work}

\paragraph{\textbf{Multiscale numerical methods}}
The numerical approximation of PDEs with highly heterogeneous and multiscale coefficients has been extensively studied in computational mathematics. Classical \emph{multiscale methods} \citep{pavliotis2008multiscale} and \emph{numerical homogenization} \citep{blanc2023homogenization} techniques aim to derive effective coarse-scale models that capture fine-scale effects implicitly. While successful in settings with clear scale separation, their performance degrades when coefficients exhibit strong local variations or lack periodicity. Localized approaches, most notably the Localized Orthogonal Decomposition (LOD) \citep{lod_book} method, address these limitations by enriching coarse finite element spaces with localized correctors that encode fine-scale information. LOD methods provide rigorous error estimates for elliptic \citep{lodoriginal} and parabolic  \citep{parab} problems with rough and high-contrast coefficients but require the solution of many localized fine-scale problems, which can be computationally expensive in many-query settings.

\paragraph{\textbf{Physics-informed machine learning}}
Physics-informed learning approaches such as \emph{Physics-Informed Neural Networks} (PINNs) \citep{RAISSI2019686} and the closely related \emph{Deep Galerkin Method} \citep{SIRIGNANO20181339} embed the governing PDEs into the training objective through residual minimization. These methods offer mesh-free approximations and flexibility with respect to geometry and boundary conditions. However, they often suffer from optimization issues and have to be retrained for every input instance.

\paragraph{\textbf{Neural operators}}
Neural operator methods aim to learn mappings between function spaces directly. \emph{DeepONet} \citep{Lu_2021} represents operators as linear combinations of a branch and trunk network, with extensions such as \emph{Physics-informed DeepONet (PI-DeepONet)} \citep{Goswami2022PhysicsInformedDN} or \emph{HyperDeepONet} \citep{lee2023hyperdeeponetlearningoperatorcomplex}. A complementary class of methods is based on spectral representations. The \emph{Fourier Neural Operator (FNO)} \citep{li2021fourierneuraloperatorparametric} and its variants, including \emph{Convolutional Neural Operators (CNO)} \citep{raonic2023convolutional}, \emph{U-shaped Neural Operators (UNO)} \citep{rahman2022unet} and the physics-informed \emph{PINO} \citep{PINO}, have demonstrated strong performance on parametric PDE benchmarks. However, their reliance on global spectral representations can limit their ability to resolve localized fine-scale features in highly heterogeneous and multiscale settings. Other work has explored attention-based and graph-based operator learning~\cite{cao2021choose,wang2025cvit,pmlr-v235-hao24d}; for example, \emph{Transolver} \citep{Transolver} is a representative example that leverages Transformer architectures to model long-range dependencies. %while \emph{Graph Neural Operator Transformers} \citep{pmlr-v202-hao23c} extend operator learning to unstructured meshes.

% \paragraph{\textbf{Hybrid approaches/Coefficient Learning}}
\paragraph{\textbf{Hybrid numerical and machine learning approaches}}
Recent work has explored hybrid approaches that combine classical numerical methods, such as finite element or spectral discretizations, with neural networks in order to leverage the strengths of both paradigms \citep{fanaskov2023spectral}. In these methods, the solution is represented in a fixed basis, and the learning task is to infer the corresponding expansion coefficients. For example, Finite element operator network (FEONet) \citep{lee2025feonet} integrates neural networks with the continuous Galerkin finite element method by directly learning solution coefficients, enabling physics-aware training without requiring labeled data. Related spectral approaches, such as SCLON \citep{Speconet} and ULGNet \citep{ulg}, learn coefficients in spectral or Galerkin bases, further demonstrating the effectiveness of combining classical numerical schemes with neural networks as a hybrid model for PDEs. Other hybrid AI-based approaches have also demonstrated substantial computational speedups and improved flexibility, with successful applications reported in areas such as solid mechanics \citep{kalina2023feann, yizheng2026pretraining}, fluid dynamics \citep{ricardo2022enhancing}, climate science \citep{gentine2018convection}, and biomedical systems \citep{peng2021multiscale}.

\section{Preliminaries}

\subsection{Problem setting and Finite Element approximation}

Let $D \subset \mathbb{R}^d$, $d \in \{2,3\}$, be a bounded polygonal (or polyhedral) Lipschitz domain with boundary $\partial D$. We consider the Darcy flow problem with homogeneous Dirichlet boundary conditions: Given a permeability field $\kappa$, find $u \colon D \to \mathbb{R}$ such that
\begin{equation}
\label{eq:darcy_prelim}
\begin{aligned}
- \nabla \cdot (\kappa \nabla u) &= f \quad &&\text{in } D, \\
u &= 0 \quad &&\text{on } \partial D,
\end{aligned}
\end{equation}
where $f \in L^2(D)$ and $\kappa \colon D \to \mathbb{R}^{d \times d}$. Throughout this work, we assume that $\kappa$ is uniformly elliptic and almost everywhere bounded. 
\begin{equation*}
0 < \operatorname*{ess\,inf}_{x\in D} \kappa(x)
\leq
\kappa(x)
\leq
\operatorname*{ess\,sup}_{x\in D} \kappa(x)
< \infty
\qquad \text{for a.e. } x\in D .
\end{equation*}
However, no further smoothness or scale separation assumptions on $\kappa$ are imposed; in particular, $\kappa$ may exhibit highly oscillatory or high-contrast behavior, i.e. the \textit{contrast} of the permeability field $\kappa$ defined by
\begin{equation*}
    \operatorname{contrast}(\kappa)
:=
\frac{
\operatorname*{ess\,sup}_{x\in D} \kappa(x)
}{
\operatorname*{ess\,inf}_{x\in D} \kappa(x)
}
\end{equation*}
might have a very high value, which increases the difficulty of solving (\ref{eq:darcy_prelim}) numerically. For the approximation of this problem using FEM-based methods, we introduce two conforming simplicial triangulations of $D$. Let $\mathcal{T}_H$ denote a coarse triangulation with maximum element diameter $H > 0$, and let $\mathcal{T}_h$ be a fine triangulation with maximum element diameter $h < H$. We assume that $\mathcal{T}_h$ is a (possibly non-uniform) refinement of $\mathcal{T}_H$ and that $\overline{D} = \bigcup_{K \in \mathcal{T}_H} K
= \bigcup_{K_h \in \mathcal{T}_h} K_h$. The fine mesh $\mathcal{T}_h$ is chosen sufficiently small to resolve all relevant variations of the diffusion coefficient $\kappa$ and is therefore referred to as a fine-scale discretization. We define the standard conforming $P_1$ finite element spaces
\begin{equation*}
P_1(\mathcal{T}) := \left\{ v \in C^0(D) \;\middle|\; v|_K \text{ is affine for all } K \in \mathcal{T} \right\},
\end{equation*}
and set
\begin{equation*}
V_h := P_1(\mathcal{T}_h) \cap H_0^1(D),
\quad
V_H := P_1(\mathcal{T}_H) \cap H_0^1(D),
\end{equation*}
with $V_H \subset V_h$. Let $\mathcal{N}_H = \{ Z_i \mid 0 \le i \le N_H - 1 \}$ and $\mathcal{N}_h = \{ z_j \mid 0 \le j \le N_h - 1 \}$
denote the sets of interior vertices (nodes) of the coarse and fine meshes, respectively, where $N_H$ and $N_h$ are their cardinalities. We denote by
\begin{equation*}
\{ \Phi_H^i \}_{i=1}^{N_H} \subset V_H,
\qquad
\{ \phi_h^j \}_{j=1}^{N_h} \subset V_h,
\end{equation*}
the associated nodal basis functions. A fine-scale finite element approximation is obtained by seeking $u_h \in V_h$ such that
\begin{equation}
\label{eq:fine_fem}
a(u_h, v_h) = \int_\Omega f v_h \, dx
\quad \text{for all } v_h \in V_h.
\end{equation}
Writing $u_h = \sum_{i=0}^{N_h-1} \boldsymbol{u}_i \phi_h^i$, this problem is equivalent to solving the linear system $\mathbf{A}\boldsymbol{u} = \mathbf{f}$ with $\boldsymbol{A}_{ij} = a(\phi_h^i,\phi_h^j)$ and $\boldsymbol{f}_i = \int_\Omega f \phi_h^i , dx$. As already indicated, in order to accurately resolve the fine-scale features of a potentially highly-varying and oscillatory coefficient $\kappa$, the mesh size $h$ must typically be chosen very small, which may result in prohibitively high computational cost for solving the linear system. On the other hand, a coarse-scale finite element approximation based on $V_H$ generally fails to capture the influence of the unresolved fine-scale variations in $\kappa$ and therefore leads to poor accuracy.

\subsection{Localized orthogonal decomposition (LOD)}

The LOD method, originally proposed in \citep{lodoriginal}, is a finite element method (FEM) designed to overcome this difficulty by constructing problem-adapted multiscale basis functions that combine the computational efficiency of coarse discretizations with the accuracy of fine-scale resolution.

\subsubsection{Orthogonal decomposition of scales}

The central idea of the LOD method is to construct an accurate approximation of the solution in a \emph{low-dimensional} space of the same dimension as a coarse finite element space $V_H$, while enriching this space with fine-scale information from a high-dimensional space $V_h$. To this end, first define a projection (or quasi-interpolation) operator $\mathcal{I}_H \colon V \to V_H$, which projects functions from the energy space $V = H_0^1(D)$ onto its associated counterpart in the coarse finite element space $V_H$. The operator $\mathcal{I}_H$ is assumed to be a projection in the algebraic sense, i.e.,
\begin{equation*}
    \mathcal{I}_H \circ \mathcal{I}_H = \mathcal{I}_H,
\end{equation*}
and to satisfy appropriate stability properties~\citep{lod_book}. Several realizations of such a projection are possible. Typical examples include
the global $L^2$-orthogonal projection onto the coarse finite element space
$V_{H}$, defined by
\[
(\mathcal{I}_H v, \Phi_H)_{L^2(\Omega)} = (v, \Phi_H)_{L^2(\Omega)}
\quad \text{for all } \Phi_H \in V_{H},
\]
as well as locally defined $L^2$-projections based on nodal patches. These local
constructions exploit projections onto polynomial spaces restricted to coarse
neighborhoods and lead to operators that satisfy the required stability properties
on quasi-uniform meshes. In contrast, nodal interpolation operators generally fail
to provide sufficient stability in the $H^1$-norm and are therefore not suitable
for the LOD framework; see \citep{lod_book} for a detailed discussion of admissible choices.  We define the space of microscopic fine-scale functions as
\begin{equation*}
W := \ker(\mathcal{I}_H)
= \left\{ v \in V \;\middle|\; \mathcal{I}_H(v) = 0 \right\}.
\end{equation*}
The space $W$ contains all functions that encode fine-scale features that cannot be represented in the coarse space $V_H$ and yields an orthogonal decomposition $V = V_H \oplus W$ w.r.t. the standard inner product on $V$. In order to construct an improved approximation space compared to $V_H$, we exploit the \textit{problem-specific inner product} on $V$ induced by the bilinear form
\begin{equation*}
a(u,v) := \int_D \kappa \nabla u \cdot \nabla v \, dx.
\end{equation*}
and define the improved multiscale approximation space as
\begin{equation*}
V_H^{\mathrm{ms}}
:= \left\{ v_H^{\mathrm{ms}} \in V \;\middle|\;
a(v_H^{\mathrm{ms}}, w) = 0 \;\; \text{for all } w \in W \right\},
\end{equation*}
which is nothing else but the orthogonal complement of the subspace $W$ in $V$ with respect to the inner product $a$, leading to the $\mathcal{A}$-orthogonal decomposition $V = V_H^{\mathrm{ms}} \oplus W$. By construction, $V_H^{\mathrm{ms}}$ has the same dimension as $V_H$ and consists of functions that are $\mathcal{A}$-orthogonal to all unresolved fine-scale components. Based on this decomposition, we define the \emph{ideal numerical homogenization method} as follows: given $f \in L^2(D)$, find $u_H^{\mathrm{ms}} \in V_H^{\mathrm{ms}}$ such that
\begin{equation}
a(u_H^{\mathrm{ms}}, v)
= \int_D f v \, dx
\quad \text{for all } v \in V_H^{\mathrm{ms}}.
\end{equation}
It is ideal in the sense that $V^{ms}_H$ contains all coarse-resolvable fine-scale features of the multiscale coefficient $\kappa$.

\subsubsection{Practical realization via localization}

\paragraph{Finite-dimensional realization of the LOD method}

The ideal numerical homogenization method introduced above is formulated in the infinite-dimensional multiscale space $V_H^{\mathrm{ms}} \subset V$. In order to make this formulation computationally feasible, we next clarify how the corresponding variational problem can be solved in practice. To this end, consider the \emph{correction operator} $Q \colon V \to W$, defined as the $a(\cdot,\cdot)$-orthogonal projection onto the detail space $W = \ker(\mathcal{I}_H)$. More precisely, for a given $v \in V$, the correction $Qv \in W$ is defined as the unique solution of
\begin{equation}
\label{correction_Def}
a(Qv, w) = a(v, w)
\quad \text{for all } w \in W.
\end{equation}
By construction, $(1 - Q)v$ is $a$-orthogonal to $W$. One can show that the operator $(1 - Q) \colon V_H \to V_H^{\mathrm{ms}}$ is bijective, with inverse given by the projection $\mathcal{I}_H$. Consequently, every function $v_H^{\mathrm{ms}} \in V_H^{\mathrm{ms}}$ can be written uniquely as $v_H^{\mathrm{ms}} = (1 - Q) v_H$ for some $v_H \in V_H$. Using this observation, the ideal multiscale problem can be reformulated as follows: find $u_H \in V_H$ such that
\begin{equation*}
a\bigl((1 - Q) u_H, (1 - Q) v_H\bigr)
= \int_D f (1 - Q) v_H \, dx
\end{equation*}
for all $v_H \in V_H$. This formulation is more favorable from a computational perspective, as it depends on the discrete, low-dimensional unknown $u_H \in V_H$ instead of the infinite-dimensional object $u_H^{\mathrm{ms}} \in V_H^{\mathrm{ms}}$. 
%In particular, by linearity, the previous equation means that the approximation is sought in the space
%\begin{equation}
%V_H^{\mathrm{ms}} = \operatorname{span}
%\left\{ \Phi_H^i - Q(\Phi_H^i) \;\middle|\; Z_i \in \mathcal{N}_H \right\},
%\end{equation}
%which consists of the coarse nodal basis functions $\Phi_H^{i}$ corrected by fine-scale information from $Q(\Phi_H^i)$. 

\paragraph{Practical computation via localization}
In practice, the numerical construction of the correction operator $Q$ itself is localized to small patches around each coarse element in order to make the method computationally feasible. This localization is motivated by the exponential decay of the corrector functions away from their associated coarse elements; see Chapter 4 in \citep{lod_book}. For a given coarse element \(T \in \mathcal{T}_H\), we define a patch \(\mathcal{N}^\ell(T)\) consisting of all coarse elements within \(\ell\) layers of \(T\) recursively by
\begin{align*}
\mathcal{N}^0(T) &:= T, \\
\mathcal{N}^\ell(T)
&:= \bigcup \left\{
T_H \in \mathcal{T}_H \;\middle|\;
T_H \cap \mathcal{N}^{\ell-1}(T) \neq \emptyset
\right\},
\quad \ell \in \mathbb{N}.
\end{align*}
and compute the localized correction operator \(Q_h^\ell\) by solving fine-scale elliptic problems of the form \eqref{correction_Def} on these patches.
To be precise, let $W^\ell(T) \subset W$ be the space consisting of functions in $W$ supported on $\mathcal{N}^\ell(T)$ and extended by zero
outside this region. The localized element corrector
\[
Q_T^\ell : V \rightarrow W^\ell(T)
\]
is then defined as the solution of
\begin{equation*}
a(Q_T^\ell v, w)
= \int_T \kappa \nabla v \cdot \nabla w \, dx
\quad \text{for all } w \in W^\ell(T).
\label{eq:localized_element_corrector}
\end{equation*}

For practical computations, all spaces are discretized using the fine-scale finite
element space $V_h$. For a given $v_h \in V_h$, the fully discrete localized element
corrector
\[
Q_{T,h}^\ell(v_h) \in W^\ell(T) \cap V_h
\]
is defined as the solution of
\begin{equation}
\label{eq:discrete_localized_corrector}
a\bigl(Q_{T,h}^\ell(v_h), w_h\bigr)
= \int_T \kappa \nabla v_h \cdot \nabla w_h \, dx
\quad \text{for all } w_h \in W^\ell(T) \cap V_h.
\end{equation}
The corresponding discrete localized correction operator is then given by
\begin{equation*}
Q_h^\ell := \sum_{T \in \mathcal{T}_H} Q_{T,h}^\ell.
\end{equation*}

Finally, the fully discrete Localized Orthogonal Decomposition method reads as
follows: given $f \in L^2(D)$, find
$u_{H,\ell,h}^{\mathrm{ms}} \in V_{H,\ell}^{\mathrm{ms},h}$ such that
\begin{equation}
\label{eq:variational_lod_method}
a(u_{H,\ell,h}^{\mathrm{ms}}, v)
= \int_D f v \, dx
\quad \text{for all } v \in V_{H,\ell}^{\mathrm{ms},h},
\end{equation}
where the multiscale space is defined by
\begin{equation*}
V_{H,\ell}^{\mathrm{ms},h}
:= \operatorname{span}
\left\{
\Phi_H^i - Q_h^\ell(\Phi_H^i)
\;\middle|\; Z_i \in \mathcal{N}_H
\right\}.
\end{equation*}
This space has the same dimension as the coarse finite element space $V_H$, and the
method can be interpreted as a Galerkin approximation of the original problem in a
problem-adapted multiscale space. Both the correction operator $Q_h^\ell$ and the
final solution $u_{H,\ell,h}^{\mathrm{ms}}$ can be obtained by assembling and solving
the corresponding linear systems; see \citep{ENGWER2019123} for implementation
details. In what follows, we shall omit the coarse and fine mesh sizes $H$ and $h$ and the localization parameter $\ell$ and call the solution $u_{H,\ell,h}^{\mathrm{ms}}$ to (\ref{eq:variational_lod_method}) simply $u_{\mathrm{LOD}}$. 

%The global localized correction operator is obtained by aggregating the contributions over all coarse elements: $\displaystyle Q_h^\ell := \sum_{T \in \mathcal{T}_H} Q_{T,h}^\ell$. Using this operator, the fully discrete multiscale finite element space is defined as $V_{H,\ell}^{\mathrm{ms},h} =
%\operatorname{span}
%\left\{
%\Phi_H^j - Q_h^\ell(\Phi_H^j)
%\right\}$ where \(\{\Phi_H^j\}\) denotes the standard coarse finite element basis. This space has the same dimension as the coarse space but is enriched with fine-scale information induced by the heterogeneous coefficient. The final LOD solution is obtained by performing a Galerkin projection of the original problem onto this problem-adapted multiscale space.

\subsection{Limitations of the LOD method}
\label{sec:lod_limitations}

Despite its effectiveness for multiscale problems without any assumption about the periodicity of $\kappa$, the LOD method suffers from notable computational limitations. In particular, LOD-based discretization requires the solution of $|\mathcal{T}_H|$ fine-scale auxiliary problems in overlapping subdomains $\mathcal{N}^\ell(T)$ for each realization of the random coefficient. As the localization parameter $\ell$ increases, these subdomains grow in size and overlap more strongly, leading to a rapidly increasing computational cost for the construction of the correction operator $Q_h^{\ell}$. However, theoretical accuracy results typically require choosing $\ell \approx C \lvert \log H \rvert$ for some constant $C>0$, which may result in relatively large localization patches in practice. This renders the computation of the correctors a major computational bottleneck. 

Moreover, the corrected basis functions lose the local support of standard finite element nodal basis functions: Instead of being supported on a single coarse element $T$, the LOD basis functions have support on the entire patch $\mathcal{N}^\ell(S)$, thereby significantly polluting the sparsity structure of the resulting system matrix. This effect is further exacerbated in settings where the resulting sparsity-polluted linear systems must be solved repeatedly. These considerations motivate the development of fully surrogate-based models that bypass both the explicit computation of the basis corrections $Q_{h}^{\ell}$ and the solution of the associated dense or weakly sparse linear systems.

\section{Proposed method: LOD Multiscale Neural Operator (LOD-MSNO)}

The overall goal of LOD-MSNO is to learn the solution operator
\[
\mathcal{G}: \kappa \mapsto u,
\]
for the multiscale Darcy problem by using neural networks to surrogate the computationally expensive components of the LOD method. We assume the permeability is a random field $\kappa=\kappa(\omega)$ on a probability space $(\Omega,\mathcal{F},\mathbb{P})$, where $\Omega$ is assumed to be compact and $\omega\in\Omega$ parameterizes one realization of the coefficient field. In the experiments, $\kappa$ may be generated from Gaussian or lognormal random fields; the specific input distributions we used are described in Section~\ref{sec:experiments}. Accordingly, both the input permeability field and solution are interpreted as Bochner-type functions written as $\kappa(x,\omega)$ and $u(x,\omega)$ for $(x,\omega) \in D \times \Omega$.

%The neural network takes the sampled coefficient realization, equivalently its parameter $\omega$, as input and returns a surrogate of the corresponding LOD representation of $u(\cdot,\omega)$.

%The LOD method achieves high accuracy for elliptic problems with rough and high-contrast coefficients by constructing problem-adapted coarse spaces. 
Central to the LOD method are the \emph{corrected FEM basis functions}
\begin{equation}
\label{eq:lod_basis}
\Psi_H^j(\omega) := \Phi_H^j - Q_{h,\kappa(\omega)}^{\ell}(\Phi_H^j), \qquad j = 1,\dots,N_H,
\end{equation}
where $\{\Phi_H^j\}$ denotes the standard coarse nodal basis functions and $Q_{h,\kappa(\omega)}^{\ell}$ is the localized correction operator computed from fine-scale problems posed on overlapping patches of diameter proportional to the localization parameter $\ell$. %As they are explicitly adapted to the permability field $\kappa(\omega)$ via (\ref{eq:discrete_localized_corrector}), the basis functions also depend on the realization through $\omega$. 
The dependence on $\omega$ enters through the bilinear form
\[
a_{\omega}(v,w) := \int_D \kappa(x,\omega)\nabla v\cdot\nabla w\,dx,
\]
which defines the localized corrector problems. Consequently, the corrected basis functions $\Psi_H^j(\omega)$ are problem-adapted and depend on the realization of the coefficient field. Using the corrected basis $\{\Psi_H^j(\omega)\}$, the LOD solution is represented as
\begin{equation}
\label{lod_linearcombi}
u_{\mathrm{LOD}}(x,\omega) = \sum_{j=1}^{N_H} \boldsymbol{u}_{\mathrm{LOD}}^{(j)}(\omega) \Psi_H^j(x,\omega),
\end{equation}
where the coefficient vector $\boldsymbol{u}_{\mathrm{LOD}}(\omega) \in \mathbb{R}^{N_H}$ is obtained by solving the linear system
\begin{equation}
\label{eq:lod_linear_eq}
    \boldsymbol{A}_{\mathrm{LOD}}(\omega) \boldsymbol{u}_{\mathrm{LOD}}(\omega) = \boldsymbol{f}_{\mathrm{LOD}}(\omega).
\end{equation}
Here, $\boldsymbol{A}_{\mathrm{LOD}}(\omega)$ is the LOD stiffness matrix and $\boldsymbol{f}_{\mathrm{LOD}}(\omega)$ is the LOD load vector, with entries
\[
{\boldsymbol{A}_{\mathrm{LOD}}(\omega)}_{ij}
:= a_{\omega}(\Psi_H^i(\omega), \Psi_H^j(\omega)),
\qquad
{\boldsymbol{f}_{\mathrm{LOD}}(\omega)}_i
:= \int_D f\Psi_H^i(x,\omega)\,dx.
\]
LOD-MSNO targets this parameter-to-solution map by predicting the coefficient vector $\boldsymbol{u}_{\mathrm{LOD}}(\omega)$ directly with a neural network, while avoiding explicit assembly and solution of \eqref{eq:lod_linear_eq} at inference time. At the same time, surrogating the construction of the correction operator $Q_{h,\kappa(\omega)}^{\ell}$---and thereby obtaining the corrected multiscale basis functions $\Psi_H^i(\omega)$---poses a separate and equally important challenge.

\subsection{LOD-MSNO (Coeff): Learning the coefficients}

\begin{figure*}[t]
  \centering
  \includegraphics[width=\textwidth]{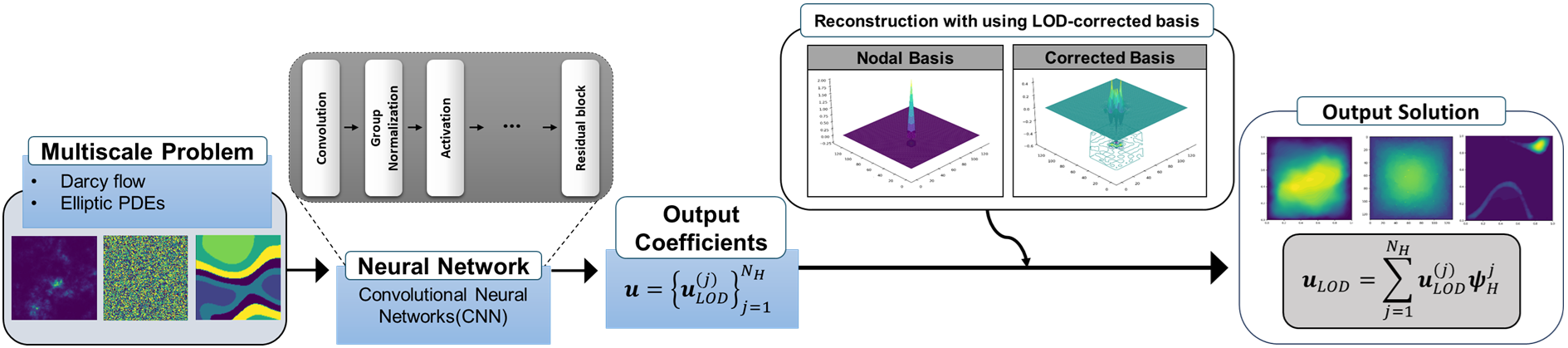}
  \caption{Schematic of the proposed LOD-MSNO. Neural network predicts the coarse LOD coefficient vector.}
  \label{fig:wide}
\end{figure*}

The first component of our surrogate focuses on learning the coefficient map
\[
\omega \mapsto \boldsymbol{u}_{\mathrm{LOD}}(\omega),
\]
or equivalently $\kappa(\cdot,\omega)\mapsto \boldsymbol{u}_{\mathrm{LOD}}(\omega)$, without assembling and solving \eqref{eq:lod_linear_eq} explicitly at inference time. To train this neural network surrogate, we draw i.i.d. realizations $\omega_i\in\Omega$ and evaluate the corresponding permeability fields $\kappa_i:=\kappa(\cdot,\omega_i)$. We then construct a dataset of tuples
\[
\{( \boldsymbol{A}_{\mathrm{LOD}}(\omega_i),
\boldsymbol{u}_{\mathrm{LOD}}(\omega_i) )\}_{i=1}^M,
\]
where $\boldsymbol{A}_{\mathrm{LOD}}(\omega_i)$ denotes the LOD stiffness matrix and $\boldsymbol{u}_{\mathrm{LOD}}(\omega_i)$ denotes the ground-truth coefficient vector for the realization $\kappa(\cdot,\omega_i)$. Since the space $V_{H,\ell}^{\mathrm{ms},h}(\omega_i) =
\operatorname{span}
\left\{
\Psi_H^j(\omega_i)
\right\}$ is typically low-dimensional, storing the matrices $\boldsymbol{A}_{\mathrm{LOD}}(\omega_i)$ remains feasible in terms of memory consumption. Moreover, these matrices have to be computed during data generation to obtain the ground-truth coefficients $\boldsymbol{u}_{\mathrm{LOD}}(\omega_i)$.
%The neural network is trained to predict an approximation $\widehat{\boldsymbol{u}}_{\mathrm{LOD}}(\kappa_i)$ of the coefficient vector directly from the multiscale input data $\kappa_i$
Rather than employing a standard $L^2$- or mean-square-based regression loss, we adopt an energy-based loss that reflects the variational structure of the underlying elliptic problem. At the population level, for a predicted coefficient map $\boldsymbol{u}:\Omega\to\mathbb{R}^{N_H}$, we define
\begin{equation}
\label{eq:energy_loss_population}
\mathcal{L}_{E}(\boldsymbol{u})
:=
\left\|
\boldsymbol{A}_{\mathrm{LOD}}^{1/2}(\omega)
\bigl(\boldsymbol{u}(\omega)-\boldsymbol{u}_{\mathrm{LOD}}(\omega)\bigr)
\right\|_{L^1(\Omega)}.
\end{equation}
In practice, this expectation over the probability space is replaced by a Monte Carlo approximation using i.i.d. samples $\omega_i\sim\mathbb{P}$, which yields the empirical loss
\begin{equation}
\label{eq:energy_loss}
\mathcal{L}_{E,M}(\boldsymbol{u})
:=
\frac{1}{M}\sum_{i=1}^M
\left|
\boldsymbol{A}_{\mathrm{LOD}}^{1/2}(\omega_i)
\bigl(\boldsymbol{u}(\omega_i)-\boldsymbol{u}_{\mathrm{LOD}}(\omega_i)\bigr)
\right|.
\end{equation}

The choice of the energy-based loss \eqref{eq:energy_loss} is directly motivated by the variational formulation of the underlying elliptic PDE. At the discrete level, this variational principle is associated with the (discrete) energy norm of a coefficient vector $\boldsymbol{v}$, which is given by
\[
\|\boldsymbol{v}\|_{\boldsymbol{A}_{\mathrm{LOD}}}^2
=
\boldsymbol{v}^{\top} \boldsymbol{A}_{\mathrm{LOD}} \boldsymbol{v}.
\]
Consequently, the proposed loss \eqref{eq:energy_loss} measures the error between the predicted and true coefficients in the natural energy norm associated with the elliptic operator. We hypothesize that employing this energy-based loss can be beneficial for generalization, as it is more closely aligned with the variational principle underlying the PDE. 
We empirically observe that training with the energy-based loss \eqref{eq:energy_loss} leads to improved accuracy and better generalization, particularly for coefficients exhibiting strong heterogeneity and high contrast. An ablation study comparing different loss functions is provided in Section~\ref{sec:ablations}. In the next section, we develop an error estimate for the LOD-MSNO (Coeff) model and analyze how using the energy-based loss affects its convergence properties.

\subsection{Error analysis for LOD-MSNO (Coeff)}
\label{subsec:theoretical_analysis}

Throughout this subsection, we assume that the corrected LOD basis functions
$\{\Psi_H^j(\omega)\}_{j=1}^{N_H}$ are available for each realization
$\omega\in\Omega$.
% Thus, the learning task considered here concerns only the coefficient map in the fixed, coefficient-dependent LOD basis.
Hence, the basis functions are treated as given, and the learning task reduces to approximating the coefficient map associated with the LOD representation.
Let $\mathcal{N}$ be any class of $\mathbb{R}^{N_H}$-valued functions parameterized by neural networks. We distinguish the following objects:
\begin{itemize}
    \item $u(x,\omega)$ denotes the true solution of the Darcy problem~\eqref{darcy_eq}.
    \item $u_{\mathrm{LOD}}(x,\omega)$ denotes the numerical LOD solution obtained as a linear combination of the LOD basis by solving the linear system~\eqref{eq:lod_linear_eq}.
    \item $\widehat{u}_{\mathrm{LOD}}^{E}(x,\omega)$ denotes the solution reconstructed from coefficients predicted by a network trained with the population loss $\mathcal{L}_{E}$ in~\eqref{eq:energy_loss_population}.
    \item $\widehat{u}_{\mathrm{LOD}}^{E,M}(x,\omega)$ denotes the actual LOD-MSNO (Coeff) prediction reconstructed from coefficients trained with the empirical Monte Carlo loss $\mathcal{L}_{E,M}$ in~\eqref{eq:energy_loss}.
\end{itemize}
More precisely, if
\[
\widehat{\boldsymbol{u}}_{\mathrm{LOD}}^{E}
\in
\operatorname*{argmin}_{\boldsymbol{u}\in\mathcal{N}}\mathcal{L}_{E}(\boldsymbol{u}),
\qquad
\widehat{\boldsymbol{u}}_{\mathrm{LOD}}^{E,M}
\in
\operatorname*{argmin}_{\boldsymbol{u}\in\mathcal{N}}\mathcal{L}_{E,M}(\boldsymbol{u}),
\]
then the corresponding reconstructed LOD solutions are
\begin{align*}
\widehat{u}_{\mathrm{LOD}}^{E}(x,\omega)
&:=
\sum_{j=1}^{N_H}
\widehat{\boldsymbol{u}}_{\mathrm{LOD}}^{E,(j)}(\omega)\Psi_H^j(x,\omega),\\
\widehat{u}_{\mathrm{LOD}}^{E,M}(x,\omega)
&:=
\sum_{j=1}^{N_H}
\widehat{\boldsymbol{u}}_{\mathrm{LOD}}^{E,M,(j)}(\omega)\Psi_H^j(x,\omega).
\end{align*}
In particular, throughout this subsection, bold symbols denote coefficient vectors, whereas non-bold symbols denote the corresponding (reconstructed) functions. The goal of the following analysis is to quantify the discrepancy between the true solution and the learned LOD-MSNO (Coeff) prediction. To this end, we decompose the total error into three contributions:
\begin{equation}
\label{eq:error_decomposition}
\begin{aligned}
u(x,\omega)-\widehat{u}_{\mathrm{LOD}}^{E,M}(x,\omega)
={}&
\underbrace{u(x,\omega)-u_{\mathrm{LOD}}(x,\omega)}_{\text{(I) LOD approximation error}} \\
&+\underbrace{u_{\mathrm{LOD}}(x,\omega)-\widehat{u}_{\mathrm{LOD}}^{E}(x,\omega)}_{\text{(II) population learning error}} \\
&+\underbrace{\widehat{u}_{\mathrm{LOD}}^{E}(x,\omega)-\widehat{u}_{\mathrm{LOD}}^{E,M}(x,\omega)}_{\text{(III) empirical/generalization error}}.
\end{aligned}
\end{equation}
The three terms correspond to the LOD discretization error, the approximation error induced by the neural network class, and the generalization error arising from finite sampling. The first term is controlled by the classical LOD approximation theory. We state the corresponding result here:
\begin{Theorem}[LOD approximation error~\citep{lod_book}]
\label{thm:lod_approximation_error}
Assume that the coefficient is uniformly elliptic, i.e.\ $0<\alpha\le \kappa(x,\omega)\le \beta<\infty$. In particular, choosing $\ell \gtrsim |\log H|$ recovers the convergence rates of the ideal method. More precisely,
\begin{equation}
\label{eq:lod_l2_main}
\|u(x,\omega)-u_{\mathrm{LOD}}(x,\omega)\|_{L^2(D)}
\lesssim
H^2,
\end{equation}
and
\begin{equation}
\label{eq:lod_energy_main}
\begin{aligned}
&\|A^{1/2}\nabla\bigl(u(x,\omega)-u_{\mathrm{LOD}}(x,\omega)\bigr)\|_{L^2(D)} \\
&\quad\le
\frac{C_I}{\sqrt{\alpha}}
{\left(
H + C_{II}\bigl(\ell^{d/2}+1\bigr)
\left(\frac{\beta}{\alpha}\right)^{3/2}
\exp\!\left(-c\,\frac{\alpha}{\beta}\,\ell\right)
\right)}^2
\|f\|_{L^2(D)}.
\end{aligned}
\end{equation}
\end{Theorem}
The proof of the above result can be found in \citep{lod_book}. In particular, term~(I) in~\eqref{eq:error_decomposition} converges like $H^2$ in the $L^2(D)$ norm when $\ell\gtrsim |\log H|$. In the following, we now develop bounds for terms~(II) and~(III) in Theorems~\ref{thm:energy_approx} and~\ref{thm:energy_gen}, respectively. Their proofs are given in the appendix. We first start with a bound for~(II).

\begin{Theorem}[Approximation error for the energy-based loss]
\label{thm:energy_approx}

Denote by $\lambda_{\min}(\omega)$ and $\lambda_{\max}(\omega)$ the smallest and biggest eigenvalues of $\boldsymbol{A}_{\mathrm{LOD}}$ respectively. Assume that there exists a positive constant $c_0 > 0$ such that
\begin{equation*}
\operatorname*{ess\,inf}_{\omega\in\Omega}\lambda_{\min}(\omega)\ge c_0>0.
\end{equation*}
Denote by $\boldsymbol{u}_{\mathrm{LOD}} \colon \Omega \longrightarrow \mathbb{R}^{N_H}$ the true coefficient map of the LOD approximation, i.e.\ the solution of the linear system~\eqref{eq:lod_linear_eq} for a given $\omega \in \Omega$
%, and use the population energy loss $\mathcal{L}_{E}$ from~\eqref{eq:energy_loss_population}.
Let $\mathcal N$ be any neural network class, and define
\[
{\boldsymbol{u}}_{\mathrm{LOD}}^E
\in
\operatorname*{argmin}_{\boldsymbol{u}\in\mathcal N}\mathcal L_E(\boldsymbol{u}).
\]
Then
\[
\|{\boldsymbol{u}}_{\mathrm{LOD}}-{\boldsymbol{u}}_{\mathrm{LOD}}^E\|_{L^1(\Omega)}
\le
c_0^{-1/2}
\inf_{\boldsymbol{u}\in\mathcal N}
\|A_{\mathrm{LOD}}^{1/2}(\boldsymbol{u}-\boldsymbol{u}_{\mathrm{LOD}})\|_{L^1(\Omega)}.
\]
If, in addition,
\[
\|\sqrt{\lambda_{\max}}\|_{L^\infty(\Omega)} := c_1 <\infty,
\]
for some $c_1>0$, then
\[
\|{\boldsymbol{u}}_{\mathrm{LOD}}-{\boldsymbol{u}}_{\mathrm{LOD}}^E\|_{L^1(\Omega)}
\le
\frac{c_1}{\sqrt{c_0}}
\inf_{\boldsymbol{u}\in\mathcal N}
\|(\boldsymbol{u}-\boldsymbol{u}_{\mathrm{LOD}})\|_{L^1(\Omega)}
\]
\end{Theorem}
Next, we shall give a bound for term~(III).
\begin{Theorem}[Generalization error for the energy-based loss]
\label{thm:energy_gen}
Assume that
\[
\operatorname*{ess\,inf}_{\omega\in\Omega}\lambda_{\min}(\omega)\ge c_0>0,
\qquad
\operatorname*{ess\,sup}_{\omega\in\Omega}\lambda_{\max}(\omega)\leq c_1 <\infty.
\]
Let $\omega_1,\dots,\omega_M$ be i.i.d.\ samples in $\Omega$.
Define the empirical energy loss as the Monte-Carlo approximation of the energy loss
\[
\mathcal L_{E,M}(\boldsymbol{u})
:=
\frac{|\Omega|}{M}\sum_{i=1}^M
\left|
\boldsymbol{A}_{\mathrm{LOD}}(\omega_i)^{1/2}
\bigl({\boldsymbol{u}}_{\mathrm{LOD}} (\omega_i)-\boldsymbol{u}(\omega_i)\bigr)
\right|.
\]
and let
\[
\widehat{\boldsymbol{u}}^{E,M}_{\mathrm{LOD}}
\in
\operatorname*{argmin}_{\boldsymbol{u}\in\mathcal N}\mathcal L_{E,M}(\boldsymbol{u})
\]
be a minimizer of the empirical energy loss function. Define the function class associated with the Neural Network class $\mathcal{N}$ and the energy loss
\[
\mathcal G^{E}
:=
\Bigl\{
\omega\mapsto
\left|
A^{1/2}_{\mathrm{LOD}}(\omega)
\bigl({\boldsymbol{u}}_{\mathrm{LOD}} (\omega)-\boldsymbol{u}(\omega)\bigr)
\right|
:\ \boldsymbol{u} \in \mathcal N
\Bigr\}.
\]
Then
\begin{equation*}
\begin{aligned}
&\mathbb E\!\left[
\left\|
\widehat{\boldsymbol{u}}^{E}_{\mathrm{LOD}}
-\widehat{\boldsymbol{u}}^{E,M}_{\mathrm{LOD}}
\right\|_{L^1(\Omega)}
\right] \\
&\quad\lesssim
\left\|
\frac{1}{\sqrt{\lambda_{\min}}}
\right\|_{L^\infty(\Omega)}
R_M(\mathcal G^{E}) \\
&\qquad+
\left\|\frac{1}{\sqrt{\lambda_{\min}}}\right\|_{L^\infty(\Omega)}
\|\sqrt{\lambda_{\max}}\|_{L^\infty(\Omega)}
\inf_{\boldsymbol{u}\in\mathcal N}
\|(\boldsymbol{u}-\boldsymbol{u}_{\mathrm{LOD}})\|_{L^1(\Omega)}.
\end{aligned}
\end{equation*}
Here, $R_M(\mathcal G^{E})$ denotes the empirical Rademacher complexity of the function class $\mathcal{G}^{E}$, and the expectation is taken with respect to the random training samples
$\omega_1,\dots,\omega_M$, and $R_M(\mathcal G_n^{E})$ .
\end{Theorem}

If the neural network class $\mathcal N$ is chosen appropriately through the architecture, such that the associated Rademacher complexity satisfies $R_M(\mathcal G^{E})\to0$ as $M\to\infty$, and if $\mathcal N$ is sufficiently expressive to approximate the LOD coefficient map, then the learned coefficient surrogate converges to the ground-truth LOD coefficients $\boldsymbol{u}_{\mathrm{LOD}}$. Since the final LOD-MSNO(Coeff) prediction is obtained by reconstructing the physical solution from these coefficients and the multiscale basis functions, it remains to transfer this coefficient convergence to convergence of the reconstructed solution. The following theorem establishes this result and is our main theoretical result. It provides a precise error bound between the exact solution of equation~\eqref{darcy_eq} and the LOD-MSNO (Coeff) coefficient surrogate, combining the LOD discretization error with the approximation and generalization errors of the neural coefficient model.

\begin{Theorem}[Error estimate for LOD-MSNO(Coeff)]
\label{thm:combined_error}

Let \(\lambda_{\min}(\omega)\) and \(\lambda_{\max}(\omega)\) denote the smallest and largest eigenvalues of the LOD stiffness matrix $\boldsymbol{A}_{\mathrm{LOD}}(\omega)$ for a given input representation $\omega$, and assume that there exist constants $c_0,c_1 > 0$ such that
\begin{equation*}
\begin{aligned}
\operatorname*{ess\,inf}_{\omega\in\Omega}\lambda_{\min}(\omega)&\ge c_0>0,\\
\|\sqrt{\lambda_{\max}}\|_{L^\infty(\Omega)}&\le c_1<\infty,
\end{aligned}
\end{equation*}
and patch size chosen as \(\ell \approx |\log H|\). Let \(R_M(\mathcal G^{E})\) denote the empirical Rademacher complexity of the function class \(\mathcal G^{E}\) defined by
\[
\mathcal G^{E}
:=
\Bigl\{
\omega\mapsto
\left|
\boldsymbol{A}_{\mathrm{LOD}}^{1/2}(\omega)
\bigl(\boldsymbol{u}_{\mathrm{LOD}}(\omega)-\boldsymbol{u}(\omega)\bigr)
\right|
:\ \boldsymbol{u} \in \mathcal N
\Bigr\}.
\]
Then
\begin{equation}
\begin{aligned}
& \mathbb E\!\left[
\|u-\widehat u_{\mathrm{LOD}}\|_{L^1(\Omega;L^2(D))}
\right]
\lesssim{}
H^2 +\frac{1}{\sqrt{c_0}}\,R_M(\mathcal G^{E}) \\
&+\frac{2 c_1}{\sqrt{c_0}}\,
%\left\|\frac{1}{\sqrt{\lambda_{\min}}}\right\|_{L^\infty(\Omega)}
%\|\sqrt{\lambda_{\max}}\|_{L^\infty(\Omega)}
\inf_{\boldsymbol{u}\in\mathcal N}
\|(\boldsymbol{u}-\boldsymbol{u}_{\mathrm{LOD}})\|_{L^1(\Omega)}.
\end{aligned}
\end{equation}
Here, the expectation $\mathbb{E}$ is taken with respect to the random training samples
$\omega_1,\dots,\omega_M$.
\end{Theorem}

\begin{proof}
We decompose the total error as

\begin{equation*}
\begin{aligned}
u-\widehat u_{\mathrm{LOD}}^{E,M}
={}&
\underbrace{u-u_{\mathrm{LOD}}}_{\text{(I)}}
+
\underbrace{u_{\mathrm{LOD}}-\widehat u_{\mathrm{LOD}}^{E}}_{\text{(II)}}
+
\underbrace{\widehat u_{\mathrm{LOD}}^{E}-\widehat u_{\mathrm{LOD}}^{E,M}}_{\text{(III)}}.
\end{aligned}
\end{equation*}
By Theorem~\ref{thm:lod_approximation_error} we have that $\|u(\cdot,\omega)-u_{\mathrm{LOD}}(\cdot,\omega)\|_{L^2(D)} \lesssim H^2$ for all $\omega \in \Omega$. Using the compactness of $\Omega$ we get:
\begin{equation*}
\|u-u_{\mathrm{LOD}}\|_{L^1(\Omega;L^2(D))}
=
\int_\Omega
\|u(\cdot,\omega)-u_{\mathrm{LOD}}(\cdot,\omega)\|_{L^2(D)}
\, d\mu(\omega)
\le
\int_\Omega C H^2 \, d\mu(\omega)
\lesssim
H^2
\end{equation*}
Next, from the definition, we have that
\begin{equation*}
\begin{aligned}
\|u_{\mathrm{LOD}}-{u}^{E,M}_{\mathrm{LOD}}\|_{L^1(\Omega;L^2(D))}
&=
\int_\Omega
\left\|
\sum_{j=1}^{N_H}
\bigl(\boldsymbol{u}^{(j)}_{\mathrm{LOD}}-(\boldsymbol{\widehat{u}}_{\mathrm{LOD}}^{E,M(j)})\bigr)\Psi^j_{H}
\right\|_{L^2(D)}
\,d\omega
\nonumber\\
&\lesssim
\int_\Omega
\left(
\sum_{j=1}^{N_H}
\bigl|\boldsymbol{u}_{\mathrm{LOD}}^{(j)}-(\boldsymbol{\widehat{u}}_{\mathrm{LOD}}^{E,M(j)})\bigr|
\|\Psi_H^{j}\|_{L^2(D)}
\right)
\,d\omega
\nonumber\\
&\lesssim
\max_{1\le j\le N_H}\|\Psi_H^{j}\|_{L^2(D)}
\int_\Omega
\left(
\sum_{j=1}^{N_H}
\bigl|\boldsymbol{u}_{\mathrm{LOD}}^{(j)}-(\boldsymbol{\widehat{u}}_{\mathrm{LOD}}^{E,M(j)})\bigr|
\right)
\,d\omega
\nonumber\\
&\lesssim
\max_{1\le j\le N_H}\|\Psi_{H}^{j}\|_{L^2(D)} \cdot
\sum_{j=1}^{N_H}
\|\boldsymbol{u}_{\mathrm{LOD}}^{(j)}
-\widehat{\boldsymbol{u}}_{\mathrm{LOD}}^{E,M(j)}\|_{L^1(\Omega)} \\
& \lesssim \|\boldsymbol{u}_{\mathrm{LOD}}
-\widehat{\boldsymbol{u}}_{\mathrm{LOD}}^{E,M}\|_{L^1(\Omega)}
\end{aligned}
\end{equation*} Theorem~\ref{thm:energy_approx} implies
\begin{equation*}
\|\boldsymbol u_{\mathrm{LOD}}-\widehat{\boldsymbol u}_{\mathrm{LOD}}^{E}\|_{L^1(\Omega)}
\le
\frac{c_1}{\sqrt{c_0}}\,
\inf_{\boldsymbol{u} \in \mathcal N}
\|(\boldsymbol{u}-\boldsymbol{u}_{\mathrm{LOD}})\|_{L^1(\Omega)},
\end{equation*}
while Theorem~\ref{thm:energy_gen} gives
\begin{equation*}
\mathbb E[
\|\widehat{\boldsymbol{u}}_{\mathrm{LOD}}^{E}-{\widehat {\boldsymbol{u}}^{E,M}_{\mathrm{LOD}}}\|_{L^1(\Omega)}]
\lesssim
\frac{1}{\sqrt{c_0}}\,R_M(\mathcal G^{E})
+\frac{c_1}{\sqrt{c_0}}\,
\inf_{\boldsymbol{u} \in \mathcal N}
\|(\boldsymbol{u}-\boldsymbol{u}_{\mathrm{LOD}})\|_{L^1(\Omega)}.
\end{equation*}
Combining the three terms proves
\begin{equation*}
\mathbb E\!\left[
\|u-\widehat u_{\mathrm{LOD}}^{E,M}\|_{L^1(\Omega;L^2(D))}
\right] \lesssim
H^2
+\frac{1}{\sqrt{c_0}}\,R_M(\mathcal G^{E})+
\frac{2c_1}{\sqrt{c_0}}\,
\inf_{\boldsymbol{u} \in \mathcal N}
\|(\boldsymbol{u}-\boldsymbol{u}_{\mathrm{LOD}})\|_{L^1(\Omega)}.
\end{equation*}

\end{proof}

Theorem~\ref{thm:combined_error} establishes that the total error between the true solution of~(\ref{darcy_eq}) is primarily governed by three contributions. The first term reflects the approximation properties of the LOD method and recovers the optimal bound when $\ell \approx |\log(H)|$. The second term captures the generalization capability of the neural network class, which can be quantified via its Rademacher complexity. The third term depends on both the expressive power of the neural network class $\mathcal{N}$ and the spectral properties of the LOD stiffness matrix corresponding to a given input $\omega \in \Omega$. By picking a neural network class $\mathcal{N}$ with strong enough approximation (third term) and generalization properties (second term), and by choosing $H$ small enough (first term), we can achieve that the expected error between the true solution $u$ and the LOD-MSNO (Coeff) prediction $\widehat{u}_{\mathrm{LOD}}^{E,M}$ converges to zero.

\subsection{LOD-MSNO (Joint): Learning the basis functions and the combined model}
\label{sec:basis_learning}

The coefficient model described so far assumes that the corrected LOD basis functions $\{\Psi_H^j(\omega)\}_{j=1}^{N_H}$ are available for every input realization. However, computing these basis functions by solving localized fine-scale corrector problems of the form (\ref{eq:discrete_localized_corrector}) for every patch constitutes the major computational bottleneck of the classical LOD method. In order to remove this remaining bottleneck, we aim to also learn the coefficient-dependent basis map
\begin{equation*}
\label{eq:basis_map_learning}
    \kappa(\cdot,\omega)
    \longmapsto
    \{\Psi_H^j(\cdot,\omega)\}_{j=1}^{N_H}.
\end{equation*}
that maps an input permeability field $\kappa$ to its corresponding set of LOD-corrected basis functions. Rather than learning the multiscale basis functions directly, we exploit their
LOD construction. For each coarse basis function $\Phi_H^j$, the corresponding
localized multiscale basis function is given by
\begin{equation}
\label{eq:learned_basis_via_corrector}
    \Psi_H^j(\omega)
    :=
    \Phi_H^j - Q_h^\ell(\Phi_H^j)(\omega)
    =
    \Phi_H^j - \sum_{T\in\mathcal{T}_H}
    Q_{T,h}^\ell(\Phi_H^j)(\omega).
\end{equation}
Consequently, it is sufficient to learn the local correction operators
$Q_{T,h}^\ell$ on each patch $\mathcal{N}^\ell(T)$.

We now rewrite the discrete localized corrector problem in a form that is suitable
for supervised learning. Let $T\in\mathcal{T}_H$ be fixed and let
$\{\phi_{\ell,T,m}\}_{m=1}^{N_{\ell,h}(T)}$ denote the fine-scale nodal basis
functions associated with the fine interior degrees of freedom in the patch
$\mathcal{N}^\ell(T)$. For a given coarse basis function $\Phi_H^j$, we write the
local corrector as
\begin{equation}
\label{eq:local_corrector_coeff_expansion}
    Q_{T,h}^\ell(\Phi_H^j)(x,\omega)
    =
    \sum_{m=1}^{N_{\ell,h}(T)}
    w_{T,j,m}(\omega)\,\phi_{\ell,T,m}(x),
\end{equation}
and collect the local correction coefficients in the vector
$\boldsymbol{w}_{T,j}(\omega)\in\mathbb{R}^{N_{\ell,h}(T)}$. The defining local
corrector equation is
\begin{equation}
\label{eq:discrete_localized_corrector_learning}
a_{\omega}\bigl(Q_{T,h}^\ell(\Phi_H^j), w_h\bigr)
= \int_T \kappa(x,\omega)\nabla \Phi_H^j \cdot \nabla w_h \, dx
\quad \text{for all } w_h \in W^\ell(T) \cap V_h.
\end{equation}
The constraint $Q_{T,h}^\ell(\Phi_H^j)\in W^\ell(T)=\ker(\mathcal{I}_H)\cap
\{v:\operatorname{supp}(v)\subset\mathcal{N}^\ell(T)\}$ can be imposed algebraically
by a local constraint matrix $\boldsymbol{C}_{T}$ representing the action of
$\mathcal{I}_H$ on the patch degrees of freedom. More precisely, if
$\boldsymbol{A}_{T}(\omega)\in\mathbb{R}^{N_{\ell,h}(T)\times N_{\ell,h}(T)}$ and
$\boldsymbol{r}_{T,j}(\omega)\in\mathbb{R}^{N_{\ell,h}(T)}$ are defined by
\begin{align*}
\label{eq:local_corrector_matrix_entries}
    (\boldsymbol{A}_{T}(\omega))_{mn}
    &:=
    \int_{\mathcal{N}^\ell(T)}
    \kappa(x,\omega)\nabla\phi_{\ell,T,n}(x)\cdot
    \nabla\phi_{\ell,T,m}(x)\,dx, \\
    (\boldsymbol{r}_{T,j}(\omega))_{m}
    &:=
    \int_T
    \kappa(x,\omega)\nabla\Phi_H^j(x)\cdot
    \nabla\phi_{\ell,T,m}(x)\,dx,
\end{align*}
then the local correction coefficients are obtained from the saddle-point system
\begin{equation}
\label{eq:local_corrector_saddle_point}
\begin{pmatrix}
\boldsymbol{A}_{T}(\omega) & \boldsymbol{C}_{T}^{\top} \\
\boldsymbol{C}_{T} & \boldsymbol{0}
\end{pmatrix}
\begin{pmatrix}
\boldsymbol{w}_{T,j}(\omega) \\
\boldsymbol{\lambda}_{T,j}(\omega)
\end{pmatrix}
=
\begin{pmatrix}
\boldsymbol{r}_{T,j}(\omega) \\
\boldsymbol{0}
\end{pmatrix}.
\end{equation}

Here, $\boldsymbol{\lambda}_{T,j}(\omega)$ is a vector of Lagrange multipliers that
enforces the fine-scale constraint $\mathcal{I}_H Q_{T,h}^\ell(\Phi_H^j)=0$ on the
patch. If we let $\boldsymbol{w}_T \in \mathbb{R}^{c_d \times N_{\ell,h}}$ be correction values associated with the patch centered around $T \in \mathcal{T}_H$ for each of the $c_d$ degrees of freedom in the element $T$, the numerical construction of the LOD basis produces, for each
realization $\omega$, a collection of local coefficient matrices
\begin{equation}
\label{eq:local_corrector_tensor}
    \boldsymbol{W}(\omega)
    :=
    \bigl\{\boldsymbol{w}_{T}(\omega)
    \;:\; T\in\mathcal{T}_H\}.
\end{equation}
Equivalently, after choosing a fixed ordering of patches and local fine degrees of
freedom, this collection can be viewed as a sparse tensor whose entries are the
numbers $w_{T,j,m}(\omega)$ in~\eqref{eq:local_corrector_coeff_expansion}. Based on the above formulation, instead of trying to directly approximate (\ref{eq:basis_map_learning}), our basis-correction network is trained to approximate the map
\begin{equation}
\label{eq:correction_map_learning}
    %\mathcal{B}_{\theta_B}:
    \kappa(\cdot,\omega)
    \longmapsto
    \widehat{\boldsymbol{W}}(\omega)
    %= \bigl\{\widehat{\boldsymbol{w}}_{T,j,\theta_B}(\omega) \;:\; T\in\mathcal{T}_H,\; j=1,\ldots,N_H\bigr\}.
\end{equation}

\begin{algorithm}[!t]
\caption{Overall Training Procedure for LOD-MSNO}
\label{alg:joint_training}
\begin{algorithmic}
\Require
Training data $\{(\kappa((\omega_i)),\boldsymbol{u}_{\mathrm{LOD}}(\omega_i),
\boldsymbol{W}(\omega_i))\}_{i=1}^{M}$ with ${\lbrace \omega_i \rbrace}_{i=1}^{M}$ being an i.i.d sequence in $\Omega$.
\Statex
\State Initialize the parameters of the coefficient and basis model $(\theta_c, \theta_B)$.
\While{the coefficient validation loss has not converged}
    \State Draw a mini-batch $\mathcal{I}\subset\{1,\ldots,M\}$
    \ForAll{$i\in\mathcal{I}$}
        \State Train the coefficient model with the energy loss~\eqref{eq:energy_loss}:
        \begin{equation*}
        \mathcal{L}_{E,M}(\boldsymbol{u})
        :=
        \frac{1}{M}\sum_{i \in \mathcal{I}}
        \left|
        \boldsymbol{A}_{\mathrm{LOD}}^{1/2}(\omega_i)
        \bigl(\boldsymbol{u}(\omega_i)-\boldsymbol{u}_{\mathrm{LOD}}(\omega_i)\bigr)
        \right|.
    \end{equation*}
    \EndFor
\EndWhile
%\State Initialize the joint model with the pretrained parameters $(\theta_u,\theta_B)$.

\While{the basis validation loss has not converged}
    \State Draw a mini-batch $\mathcal{I}\subset\{1,\ldots,M\}$ and a subset $\mathcal{T} \subset \mathcal{T}_H$.
    \ForAll{$i\in\mathcal{I}, T \in \mathcal{T}$}
        \State Train the basis-correction model with the correction loss~\eqref{eq:basis_supervised_loss}:
        \begin{equation*}
            \mathcal{L}_{basis,M}(\theta_B)
            :=
            \frac{1}{M}
            \sum_{i=1}^{M}
            \sum_{T\in\mathcal{T}}
            \sum_{j=1}^{N_H}
            \frac{1}{N_{\ell,h}(T)}
            \left\|
            \boldsymbol{P}_T\widehat{\boldsymbol{w}}_{T,j}(\omega_i)
            -
            \boldsymbol{w}_{T,j}(\omega_i)
            \right\|.
        \end{equation*}
    \EndFor
\EndWhile
\State Predict LOD coefficients $\widehat{\boldsymbol{u}}_{\mathrm{LOD}}^{E,M}(\omega_i)$ from $\kappa_i$.
\State Predict local correction coefficients $\widehat{\boldsymbol{W}}(\omega_i)=\{\widehat{\boldsymbol{w}}_{T}(\omega_i):T\in\mathcal{T}_H\}$.
\State Assemble the learned correction matrix $\widehat{\boldsymbol{Q}}_h^{\ell}(\omega_i)$ from $\widehat{\boldsymbol{W}}(\omega_i)$.
\State Reconstruct the full LOD solution via (\ref{eq:full_solution_pred}).
%\If{the joint validation error has not converged}
%    \State Retrain $\theta_c$ and $\theta_B$ using the joint reconstruction loss~\eqref{eq:joint_reconstruction_loss}.
%    \State \Return $(\theta_c, \theta_B)$
%\Else
%    \State \Return $(\theta_c, \theta_B)$
%\EndIf

\end{algorithmic}
\end{algorithm}

Given a training set
\begin{equation}
\label{eq:basis_training_dataset}
    \mathcal{D}_{B}
    :=
    \bigl\{(\kappa_i,\boldsymbol{W}_i)\bigr\}_{i=1}^{M},
    \qquad
    \boldsymbol{W}_i=\boldsymbol{W}(\omega_i),
\end{equation}
where $\boldsymbol{W}_i$ is computed by solving~\eqref{eq:local_corrector_saddle_point} for all relevant patches and coarse basis functions, we use the supervised loss function
\begin{equation}
\label{eq:basis_supervised_loss}
    \mathcal{L}_{basis,M}(\theta_B)
    :=
    \frac{1}{M}
    \sum_{i=1}^{M}
    \sum_{T\in\mathcal{T}_H}
    \sum_{j=1}^{N_H}
    \frac{1}{N_{\ell,h}(T)}
    \left\|
    \boldsymbol{P}_T\widehat{\boldsymbol{w}}_{T,j}(\omega_i)
    -
    \boldsymbol{w}_{T,j}(\omega_i)
    \right\|.
\end{equation}
the learn the parameters $\theta_B$ of the basis-correction network. Here, $\boldsymbol{P}_T$ denotes the nullspace projector
\begin{equation*}
    \boldsymbol{P}_T = I - \boldsymbol{C}_T^T {(\boldsymbol{C}_T\boldsymbol{C}_T^T)}^{-1}\boldsymbol{C}_T
\end{equation*}
that ensures the predicted coefficients $\widehat{\boldsymbol{w}}_{T,j}$ are in the space $\ker(\mathcal{I}_H)\cap
\{v:\operatorname{supp}(v)\subset\mathcal{N}^\ell(T)\}$. In practice, the summation over $j$ only involves coarse basis functions whose
support intersects the element $T$, since all other right-hand sides $\boldsymbol{r}_{T,j}(\omega)$ vanish. Once the local corrections have been predicted, the learned multiscale basis is assembled by
\begin{equation}
\label{eq:learned_multiscale_basis}
    \widehat{\Psi}_{H}^j(\omega)
    :=
    \Phi_H^j
    -
    \sum_{T\in\mathcal{T}_H}
    \sum_{m=1}^{N_{\ell,h}(T)}
    \widehat{\boldsymbol{w}}_{T,j}^{(m)}(\omega)\,\phi_{\ell,T,m}.
\end{equation}
This gives a learned approximation of the full LOD basis map in
\eqref{eq:basis_map_learning} without solving local saddle-point systems at
inference time. To obtain a single neural
operator from the permeability field to the fine-scale solution, both components
are combined through the LOD reconstruction formula. Let $\widehat{\boldsymbol{u}}_{\mathrm{LOD}}(\omega)$ denote the coefficient prediction and let
$\widehat{\boldsymbol{W}}(\omega)$ denote the predicted local correction
coefficients. LOD-MSNO (Joint) reconstructs the full solution as \begin{equation}
\label{eq:full_solution_pred}
    \begin{aligned}
        \widehat{u}_{LOD-MSNO} &= \sum_{j=1}^{N_H}
\widehat{\boldsymbol{u}}_{\mathrm{LOD}}^{E,M,(j)}(\omega)\widehat{\Psi}_H^j(x,\omega) \\
&= \sum_{j=1}^{N_H}
\widehat{\boldsymbol{u}}_{\mathrm{LOD}}^{E,M,(j)}(\omega) \left( \Phi_H^j - \sum_{T\in\mathcal{T}_H}
    \widehat{Q}_{T,h}^\ell(\Phi_H^j)(\omega) \right) \\
& = \sum_{j=1}^{N_H}
\widehat{\boldsymbol{u}}_{\mathrm{LOD}}^{E,M,(j)}(\omega)\Phi_H^j - \sum_{j=1}^{N_H}\widehat{\boldsymbol{u}}_{\mathrm{LOD}}^{E,M,(j)}(\omega)\sum_{T\in\mathcal{T}_H} \sum_{m=1}^{N_{\ell,h}(T)}
    \widehat{\boldsymbol{w}}_{T,j}^{(m)}(\omega)\,\phi_{\ell,T,m}
    \end{aligned}
\end{equation}

The coefficient network and the basis-correction network define two complementary
surrogates. The first network approximates the coefficient map
$\kappa(\cdot,\omega)\mapsto\boldsymbol{u}_{\mathrm{LOD}}(\omega)$, whereas the
second network approximates the correction map
$\kappa(\cdot,\omega)\mapsto\boldsymbol{W}(\omega)$. LOD-MSNO then predicts the full solution via (\ref{eq:full_solution_pred}). The full procedure to train the full LOD-Multiscale Neural Operator is summarized in Algorithm \ref{alg:joint_training}.

\begin{comment}
\begin{remark}
The two networks are trained separately, so there might be a possibility that the errors of the two networks accumulate and lead to a nonaccurate or non-physical prediction. In what follows, we shall describe an additional, optional training stage to jointly learn the coefficient and basis network.

Denote by $\boldsymbol{U}_{LOD-MSNO}$ the evaluation of numerically computed solution $u_{\mathrm{LOD}}$ on the fine degrees of freedom $\mathcal{N}_h$ and let $\widehat{\boldsymbol{U}}_{LOD-MSNO}$ be the evaluation of the predicted solution (\ref{eq:full_solution_pred}) on $\mathcal{N}_h$. We shall use $\boldsymbol{U}_{LOD-MSNO}$ as a training target and re-train the parameters of both the coefficient and the basis model using the loss function
\begin{equation}
\label{eq:joint_reconstruction_loss}
    \mathcal{L}_{\mathrm{fine-tune},M}(\theta_u,\theta_B)
    :=
    \frac{1}{M}
    \sum_{i=1}^{M}
    \left\|
    \widehat{\boldsymbol{U}}_{LOD-MSNO}
    -
    \boldsymbol{U}_{LOD-MSNO}
    \right\|.
\end{equation}
The purpose of this additional fine-tuning step is to harmonize the two pretrained models and to make sure that the coefficient model and basis model are together producing an accurate and physically consistent prediction. This objective backpropagates through both the coefficient
prediction and the predicted correction coefficients, and therefore adjusts the two components with respect to the final quantity of interest.
\end{remark}
\end{comment}

\section{Experiments}
\label{sec:experiments}

We evaluate the proposed LOD-MSNO on the steady Darcy flow (\ref{darcy_eq}) problem, where the objective is to approximate the LOD representation of the solution from the permeability coefficient $\kappa$. As already indicated through naming, we distinguish two models. \emph{LOD-MSNO (Coeff)} is the coefficient model that assumes the exact corrected multiscale basis $\{\Psi_H^i\}$ to be available and learns the map from $\kappa$ to the expansion coefficients $\boldsymbol{u}_{\mathrm{LOD}}$ in \eqref{lod_linearcombi}. The full PDE solution is then reconstructed as a linear combination of the exact corrected basis functions with the learned coefficients. \emph{LOD-MSNO (Joint)} is the combination of the coefficient model and the basis model that predicts, by contrast, both the coefficients $\widehat{\boldsymbol{u}}_{\mathrm{LOD}}$ and the localized correction operators that define the coefficient-dependent multiscale basis as described in Section~\ref{sec:basis_learning} and reconstructs the full solution via (\ref{eq:full_solution_pred}). Thus, the coefficient model only replaces the LOD-system (\ref{eq:lod_linear_eq}) assembly and solve, whereas the basis model additionally replaces the localized basis corrector computations and predicts the solution to (\ref{darcy_eq}) without the need of any numerical computations.

\paragraph{Baseline models}
We compare our method against a range of established baselines to ensure a balanced and meaningful evaluation. First, we include FNO \citep{li2021fourierneuraloperatorparametric}, CNO \citep{raonic2023convolutional}, and UNO \citep{rahman2022unet}, which are widely used and representative neural operator architectures based on the iterative application of kernel integrals.  Similarly, Transolver \citep{Transolver} is incorporated as a multi-purpose transformer-based operator model capable of capturing long-range dependencies. These models serve as standard data-driven operator learning benchmarks. To ensure fairness with respect to physics-informed approaches, we additionally compare against PINO \citep{PINO}, as our method also incorporates physical structure—specifically through the use of problem-adapted basis functions and the inclusion of the LOD stiffness matrix $\boldsymbol{A}_{LOD}$ in the loss function. We further include FEONet \citep{lee2025feonet} as a representative hybrid neural network--numerical method, since it also follows a coefficient-learning paradigm in which the solution is reconstructed as a linear combination of basis functions. In contrast to LOD-MSNO, however, FEONet uses fixed coarse $P1$ finite element basis functions rather than coefficient-dependent corrected LOD basis functions.  Finally, we consider PI-DeepONet \citep{Goswami2022PhysicsInformedDN}, which is both physics-informed and structurally similar to our approach in that it represents the solution as a linear expansion over a learned set of basis functions. Our LOD-MSNO models are trained to predict the LOD coefficient vector $\boldsymbol{u}_{LOD}$, from which the full fine-scale solution is reconstructed using the corrected LOD basis.

\subsection{Benchmark datasets}
As depicted in Figure \ref{fig:Darcy_coeffs}, training and test datasets are generated using four distinct classes of permeability fields designed to probe different multiscale and high-contrast regimes. In the following, we refer to them as \emph{lognormal1}, \emph{lognormal2}, \emph{quantile}, and \emph{checkerboard}. These datasets are specifically constructed to highlight a key shortcoming of neural operator models when applied to multiscale problems and are inspired by industry-relevant real-world applications. A more detailed description of the dataset characteristics and its generation process can be found in~\ref{app:datasets}.
\paragraph{Lognormal coefficients}
For the lognormal datasets, we consider permeability fields of the form $\kappa(x) = \exp(Z(x))$ where $Z$ is a centered Gaussian random field with Whittle--Mat\'ern covariance with smoothness parameter $\nu>0$, correlation length $\ell>0$, and variance $\sigma^2>0$.  
We distinguish two regimes: (1) \textit{Lognormal1} with $\sigma=1.0$, $\ell=0.1$, (2) \textit{Lognormal2} with $\sigma=2.0$, $\ell=0.3$. The second setting exhibits significantly stronger oscillations and substantially higher contrast, making it particularly challenging for operator learning methods.
\paragraph{Quantile coefficients}
For the quantile dataset, we first sample a Gaussian random field $Z$ with squared-exponential covariance and define a lognormal field $\exp(Z)$. The resulting values are then discretized into a finite set of prescribed permeability levels via empirical quantiles and randomly permuted across the domain, yielding piecewise-constant coefficients with sharp jumps and high contrast.
\paragraph{Checkerboard coefficients.}
The checkerboard coefficients are defined as piecewise-constant random fields on a fine Cartesian grid, where each cell independently takes values from a fixed set of permeability levels. This yields coefficients with discontinuities at the grid scale and no spatial correlation.

\subsection{Evaluation of LOD-MSNO}

\subsubsection{Comparison of baseline models and LOD-MSNO}

\begin{figure*}[t]
  \centering
    \includegraphics[width=\linewidth]{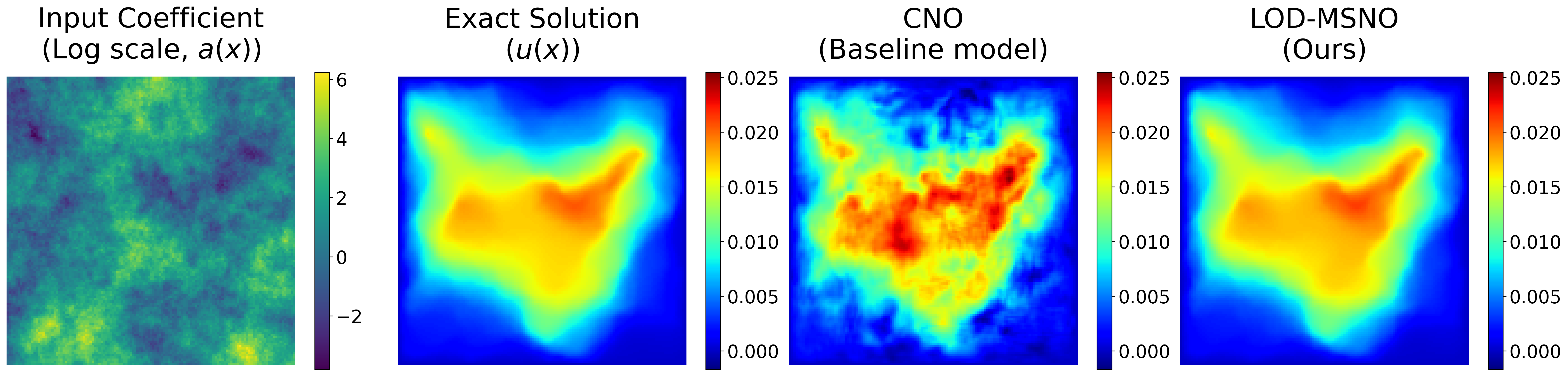}
    \caption{Prediction errors for lognormal2 coefficients using the best-performing baseline model and LOD-MSNO(Coeff).}
  \label{fig:solution_profile}
\end{figure*}

In this section, we fix the coarse and fine mesh sizes to $H=2^{-4}$ and $h=2^{-7}$ for the evaluation of LOD-MSNO.
\textit{LOD-MSNO (Coeff)} uses the exact corrected multiscale basis functions and only replaces the LOD coefficient solve with a learned coefficient predictor.
The coefficient model is realized by \texttt{GlobalCoeffNet}, which maps the normalized permeability field to the coefficient vector $\boldsymbol{u}_{\mathrm{LOD}}$ associated with the corrected multiscale basis $\{\Psi_H^{i}\}$. It first extracts fine-grid features by a convolutional encoder built from Conv--GroupNorm--GELU and residual blocks and then performs nonlinear mixing on the coarse LOD grid by residual layers, so that the output is restricted to the interior coarse degrees of freedom. Training is supervised by precomputed LOD coefficient vectors and uses the energy-based objective in \eqref{eq:energy_loss}, which weights coefficient errors by the LOD stiffness matrix and therefore reflects their contribution to the reconstructed solution. At inference time, the predicted coefficients are combined with the exact corrected basis functions to reconstruct the solution, which is then evaluated on a uniform $128\times128$ grid.

\textit{LOD-MSNO (Joint)} combines the coefficient model with a learned basis-correction model, so that both the expansion coefficients $\boldsymbol{u}_{\mathrm{LOD}}$ and the coefficient-dependent localized LOD basis ${\lbrace \Psi_H^j \rbrace}_{j=1}^{N_H}$ are predicted from the permeability field.
The basis model is realized by \texttt{PatchCorrectionNet}, which acts locally on oversampling patches and predicts the correction coefficients associated with the three local coarse basis functions of each triangular coarse element.
Because boundary and interior patches have different active fine degrees of freedom, each patch is embedded into a common rectangular patch canvas using padding, and the network receives the normalized log-permeability together with input and target support masks.
Architecturally, \texttt{PatchCorrectionNet} is a U-Net-style convolutional encoder--decoder with skip connections, using the same Conv--GroupNorm--GELU and residual blocks as the coefficient model.
The local patch predictions are assembled into a global approximation of the LOD correction operator and combined with the predicted coefficient vector to reconstruct the fine-scale solution.

All baseline models are trained directly on pairs $\{(\kappa_i, u_h^{(i)})\}$ of permeability fields and corresponding fine-scale reference solutions computed by (\ref{eq:fine_fem}) on the fine grid with mesh size $h$, using the relative $L^2$ error as the training objective. For the models LOD-MSNO, FNO, CNO, and UNO, we apply the same input preprocessing and normalization pipeline, consisting of a logarithmic transformation of the permeability field, gradient-magnitude and coordinate channels, and channel-wise input as well as per-component output normalization (see Appendix~\ref{app:arch_lodmimetic} for details). For PI-DeepONet and Transolver, which do not support additional input channels, the gradient-magnitude and coordinate augmentations (\texttt{add\_grad} and \texttt{add\_coords}) are therefore omitted. For evaluation, we report the mean and standard deviation of the relative $L^2$ error (Rel-\(L^2\)) computed on the fine grid, where $u$ denotes the reference solution and $\widehat{u}$ the model prediction. 

\begin{table*}[t]
  \caption{Mean $\pm$ standard deviation of the test relative $L^2$ errors on the fine grid with $h = 2^{-7}$ across 5 independent runs. The best-performing model is shown in \textbf{bold}, and the second-best is underlined.
}
  \label{tab:main_comparison}
  \centering
  \setlength{\tabcolsep}{7pt}
  \renewcommand{\arraystretch}{1.12}
  \resizebox{\textwidth}{!}{%
  \begin{tabular}{lcccc}
    \toprule
    Model & Lognormal1 & Lognormal2 & Quantile & Checkerboard \\
    \midrule
    FNO~\citep{li2021fourierneuraloperatorparametric} & $0.049 \pm 0.001$ & $0.163 \pm 0.002$ & $0.254 \pm 0.010$ & $0.064 \pm 0.000$ \\
    CNO~\citep{raonic2023convolutional} & $\underline{0.018 \pm 0.001}$ & $0.093 \pm 0.003$ & $\underline{0.120 \pm 0.007}$ & $\mathbf{0.010 \pm 0.001}$ \\
    UNO~\citep{rahman2022unet} & $0.044 \pm 0.000$ & $0.160 \pm 0.003$ & $0.239 \pm 0.008$ & $0.064 \pm 0.001$ \\
    PI-DeepONet~\citep{Goswami2022PhysicsInformedDN} & $0.138 \pm 0.001$ & $2.028 \pm 0.084$ & $1.282 \pm 0.015$ & $0.062 \pm 0.000$ \\
    Transolver~\citep{Transolver} & $0.082 \pm 0.022$ & $0.341 \pm 0.016$ & $0.610 \pm 0.050$ & $0.053 \pm 0.003$ \\
    PINO~\citep{PINO} & $0.025 \pm 0.000$ & $0.145 \pm 0.002$ & $0.203 \pm 0.012$ & $0.035 \pm 0.000$ \\
    FEONet~\citep{lee2025feonet} & $0.099 \pm 0.000$ & $1.805 \pm 0.073$ & $0.129 \pm 0.020$ & $0.668 \pm 0.109$ \\
    \textbf{LOD-MSNO (Coeff)} & $\mathbf{0.012 \pm 0.002}$ & $\mathbf{0.042 \pm 0.006}$ & $\mathbf{0.118 \pm 0.017}$ & $\underline{0.018 \pm 0.003}$ \\
    \textbf{LOD-MSNO (Joint)} & $0.029 \pm 0.001$ & $\underline{0.068 \pm 0.002}$ & $0.302 \pm 0.016$ & $0.029 \pm 0.001$ \\
    \bottomrule
  \end{tabular}
  }
  \vskip -0.1in
\end{table*}

The aggregated results are summarized in Table~\ref{tab:main_comparison}. The results indicate that LOD-MSNO(Coeff) attains competitive and, in three out of four cases, improved relative \(L^2\)-error accuracy compared to the considered baselines, with particularly strong performance on the multiscale lognormal datasets, while remaining comparable to the best-performing methods on the other coefficient configurations. In comparison to FEONet and PI-DeepONet, the results highlight the advantage of employing problem-adapted multiscale basis functions. While FEONet relies on a fixed coarse \(P1\)-FEM basis and PI-DeepONet learns a global set of basis functions via the trunk network, LOD-MSNO leverages LOD basis functions that are specifically tailored to the underlying coefficient field. This leads to consistently improved accuracy, particularly in strongly heterogeneous regimes. Moreover, compared to physics-informed approaches such as PINO and PI-DeepONet, the results suggest that incorporating the physical structure through a strong numerical prior is more effective than enforcing the PDE via autodifferentiation-based loss terms. In particular, constructing the solution using the LOD framework allows the model to directly encode relevant multiscale features of the problem, which appears to be a more robust way of injecting physical information than relying solely on PDE residual minimization. The results for LOD-MSNO (Joint) further show that the fully learned surrogate remains competitive despite also replacing the localized corrector computations. In particular, on the high-contrast lognormal2 dataset, LOD-MSNO (Joint) achieves a mean relative \(L^2\) error of $0.068$, outperforming all benchmark models and demonstrating the potential of the full surrogate in high-contrast regimes. 
\begin{comment}
    The larger errors on the quantile and checkerboard datasets indicate that accurately learning localized basis corrections for discontinuous fine-scale coefficients remains the main challenge for the joint model.
\end{comment}

\subsubsection{Ablation study}\label{sec:ablations}

\begin{figure}[t]
  \centering
  \includegraphics[width=0.48\linewidth]{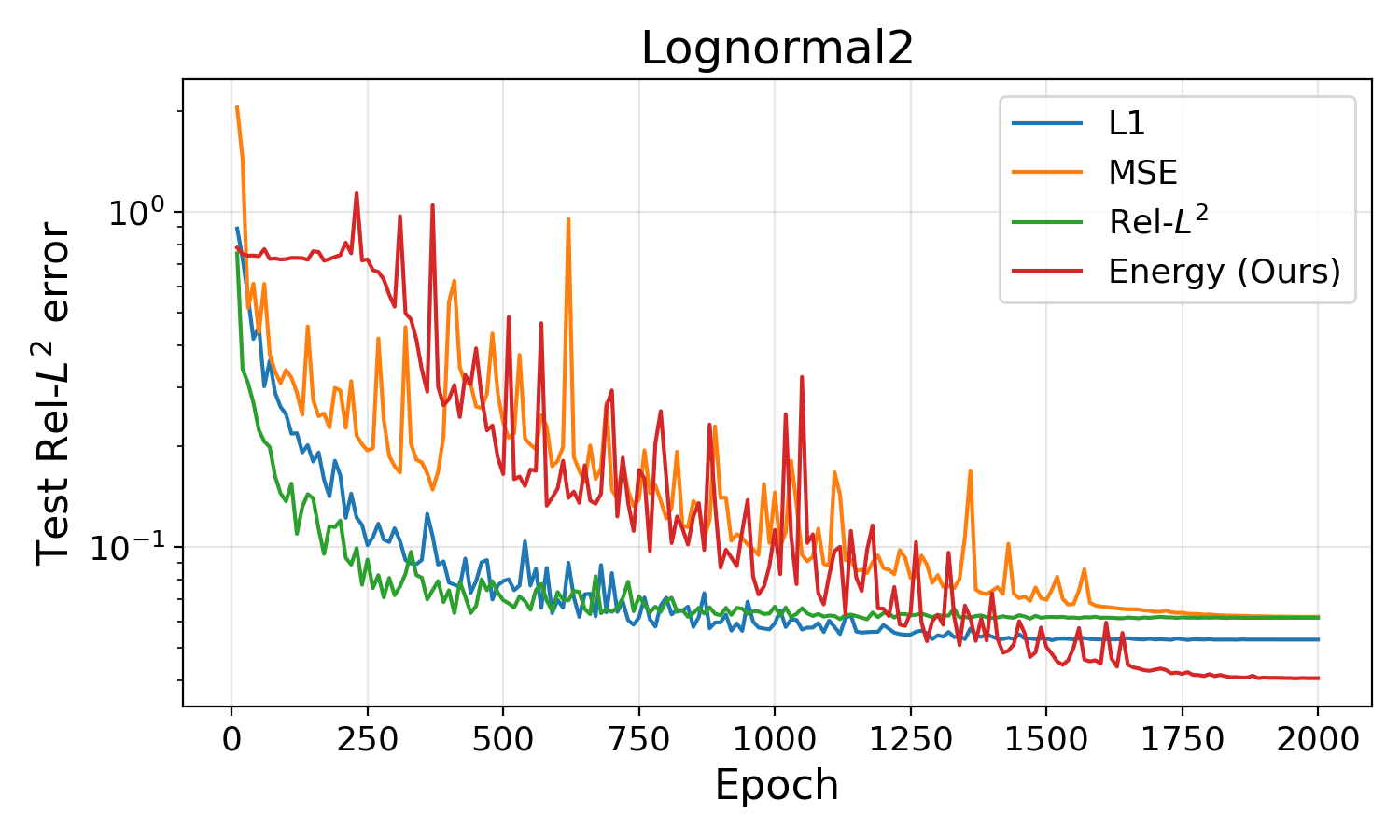}\hfill
  \includegraphics[width=0.48\linewidth]{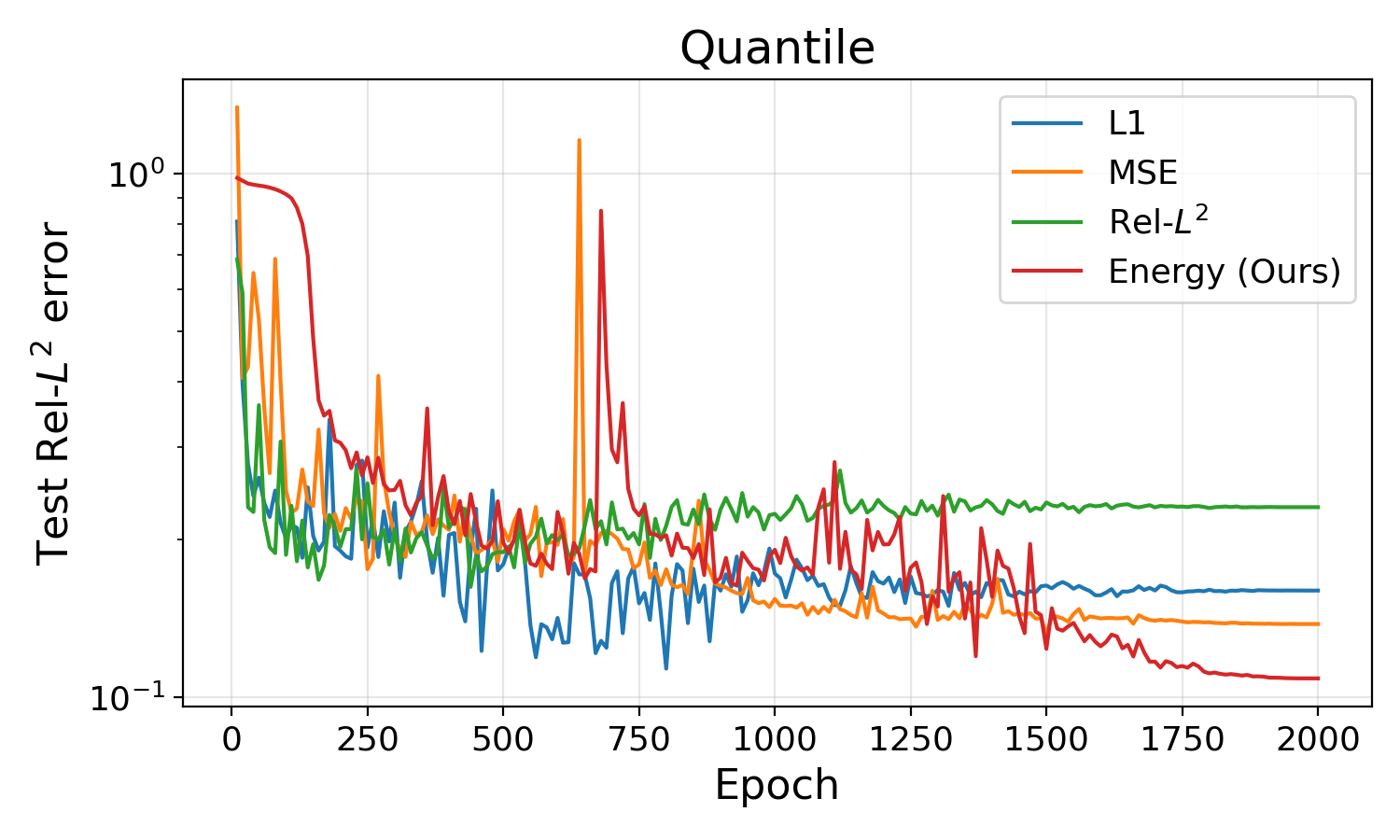}
  \caption{Loss ablation for predicting $\boldsymbol{u}_{\mathrm{LOD}}$. Test $\mathrm{Rel}\text{-}L^2$ error vs.\ epoch for lognormal2 (left) and quantile (right).}
  \label{fig:ablation_losstype}
\end{figure}

\begin{table*}[t]
  \caption{Ablation on preprocessing and input-channel augmentation for LOD-MSNO(Coeff). Fine-grid test relative $L^2$ errors.}
  \label{tab:ablation_errors}
  \centering
  \footnotesize
  \setlength{\tabcolsep}{3pt}
  \resizebox{\textwidth}{!}{%
  \begin{tabular}{llcccc}
    \toprule
    \multirow{2}{*}{Dataset} & \multirow{2}{*}{\begin{tabular}{@{}c@{}}Training\\setting\end{tabular}}
      & \multicolumn{4}{c}{Variants} \\
    \cmidrule(lr){3-6}
    & & \begin{tabular}{@{}c@{}}LOD-MSNO\\(Coeff)\end{tabular}
      & \begin{tabular}{@{}c@{}}w/o log \\+normalization\end{tabular}
      & \begin{tabular}{@{}c@{}}w/o \texttt{add\_grad} \\+\texttt{add\_coords}\end{tabular}
      & w/o all \\
    \midrule

    \multirow{2}{*}{} & {\scriptsize \begin{tabular}{@{}l@{}}log +\\normalization\end{tabular}}
      & \bluecheck & \xmark     & \bluecheck & \xmark \\
    & {\scriptsize \begin{tabular}{@{}l@{}}\texttt{add\_grad} +\\\texttt{add\_coords}\end{tabular}}
      & \bluecheck & \bluecheck & \xmark     & \xmark \\
    \midrule
    Lognormal1   &  & \textbf{0.025} & 0.080 & 0.034 & 0.236 \\
    Lognormal2   &  & \textbf{0.042} & 0.135 & 0.096 & 1.215 \\
    Quantile     &  & \textbf{0.089} & 0.247 & 0.319 & 0.313 \\
    Checkerboard &  & \textbf{0.020} & 0.020 & 0.056 & 0.188 \\
    \bottomrule
  \end{tabular}%
  }
  \vskip -0.1in
\end{table*}

We conduct an ablation study to assess the impact of the training loss function on the performance and generalization ability of LOD-MSNO (Coeff). In particular, we train the model using three alternative losses: the standard $L^2$ loss, the relative $L^2$ loss, and the $L^1$ loss. We specifically use LOD-MSNO (Coeff) to isolate the effect of the choice of loss function from LOD-MSNO (Joint), in which the error predicting the basis also influences the relative L2 error on the fine grid.
All models are trained on 500 samples with lognormal2 and quantile-type permeability fields, and we report the evolution of the test relative $L^2$ error during training. As shown in Figure \ref{fig:ablation_losstype}, training with the energy-based loss yields the lowest final error for both coefficient classes, while the $L^1$ and (relative) $L^2$ saturate at higher error levels. Especially in the lognormal2 case, the energy-loss exhibits the most stable convergence behavior, whereas the $L^2$- and $L^1$-based training runs display larger fluctuations and higher asymptotic errors. 
Table~\ref{tab:ablation_errors} reports the impact of the proposed preprocessing and channel augmentation strategies on the performance of LOD-MSNO (Coeff).  
Here, all models are trained using the energy-norm loss on 500 samples from the Quantile dataset.  
Across all datasets, the full configuration—combining logarithmic scaling, normalization, and the additional gradient-magnitude and coordinate channels—consistently yields the lowest relative $L^2$ errors. Removing either the log-scaling and normalization or the auxiliary channels leads to a noticeable degradation in accuracy, while omitting all preprocessing steps results in a substantial loss of performance. These results demonstrate that both the normalization strategy and the inclusion of gradient and coordinate information play an important role in enabling accurate and stable learning for LOD-MSNO(Coeff).

\subsection{Comparison of training and inference time}
\label{sec:time_comparison}

Table~\ref{tab:time_comparison} summarizes the data-generation costs for different mesh resolutions; all times are measured in seconds per sample.
For the LOD basis construction, we report the local patch-corrector cost, the global basis assembly cost, and their sum parallelized over 16 CPU cores, the sequential global basis assembly cost, and their sum.
The localized corrector problems $Q_{T,h}^{\ell}$ are embarrassingly parallel---each depends only on the fine-grid stiffness matrix restricted to its patch---so the patch-corrector step distributes exactly across cores; scaling to more than 16 cores would reduce this component proportionally.
By contrast, assembling the LOD basis (given by the rows of the matrix $G = P_1 - Q_h^{\ell}$, which is the difference of the prolongation
matrix P1 that encodes the coarse nodal basis functions and the correction operator $Q_h^{\ell}$) via global sparse-matrix stacking and multiplication across all patches such that is an aggregation step that does not decompose into independent patch-level tasks and is therefore measured without inter-task parallelization.
Classical LOD timings were obtained on a 16-core CPU; all neural-network timings (training and inference) were measured on an NVIDIA A100 80GB GPU.

\begin{table*}[t]
  \caption{Comparison of data-generation times, measured in seconds per sample. Local patch-corrector times assume 16 CPU cores. Global basis assembly involves global sparse matrix operations and is measured sequentially on the same CPU.}
  \label{tab:time_comparison}
  \centering
  \setlength{\tabcolsep}{3pt}
  \renewcommand{\arraystretch}{1.12}
  \resizebox{\textwidth}{!}{%
  \begin{tabular}{l ccc ccc c}
    \toprule
    \multirow{2}{*}{Resolution}
    & \multirow{2}{*}{Fine FEM}
    & \multicolumn{3}{c}{LOD Basis}
    & \multicolumn{3}{c}{LOD Coeff} \\
    \cmidrule(lr){3-5}\cmidrule(lr){6-8}
    & & \begin{tabular}{@{}c@{}}Local patch\\correctors\end{tabular} & \begin{tabular}{@{}c@{}}Global basis\\assembly\end{tabular} & Total
    & Coeff. assembly & Coeff. solve & Total \\
    \midrule
    $h=7,\ H=4$ & 0.1135 & 0.7842 & 0.4890 & 1.2732 & 0.2145 & 0.0020 & 0.2165 \\
    $h=8,\ H=4$ & 0.5590 & 4.2527 & 1.8327 & 6.0854 & 0.8500 & 0.0021 & 0.8521 \\
    $h=8,\ H=5$ & 0.5606 & 4.5842 & 4.7611 & 9.3453 & 1.2688 & 0.0752 & 1.3440 \\
    $h=9,\ H=6$ & 4.3828 & 20.0033 & 72.9008 & 92.9041 & 6.9178 & 3.1047 & 10.0225 \\
    \bottomrule
  \end{tabular}%
  }
  \vskip -0.1in
\end{table*}

The results show that the localized basis-corrector computations dominate the classical LOD data-generation cost, especially as the fine mesh is refined.
With 16-core parallelization of the local corrector step, the per-sample patch-corrector time is substantially reduced compared to the sequential cost; using more than 16 cores would decrease this cost further proportionally.
The global basis assembly and the coefficient steps remain sequential, and for fine coarse spaces, the global assembly itself becomes a dominant cost---at $h=9$, $H=6$, it accounts for the majority of the basis construction time.
In particular, as $H$ is refined, the coarse LOD solve becomes comparable in cost to the fine FEM solve despite involving fewer degrees of freedom, which highlights the effect of the non-sparsity of the LOD system as a major computational bottleneck.
This observation motivates the coefficient surrogate, whose direct inference of the expansion coefficients avoids repeatedly assembling and solving this dense coarse system and thereby contributes directly to the speedup of LOD-MSNO.
These timings motivate the two surrogate components of LOD-MSNO: the basis model targets the expensive localized corrector computations, whereas the coefficient model targets the assembly and solution of the coarse LOD system.
Consequently, the timing results indicate that LOD-MSNO indeed addresses the main computational bottlenecks of the classical LOD method discussed in Section~\ref{sec:lod_limitations}.

Table~\ref{tab:training_inference_comparison} reports the corresponding neural-network training and single-sample inference times for LOD-MSNO and CNO.
All neural-network models were trained using 1000 training samples, and the coefficient model and CNO were trained for the same number of epochs.
We compare against CNO because it is the strongest benchmark model in terms of accuracy among the standard neural-operator baselines considered in Table~\ref{tab:main_comparison}.

\begin{table*}[t]
  \caption{Comparison of training and inference times, measured in seconds. Inference times are reported per sample using a batch size of one. Training times use 500 training samples, with the coefficient model and CNO trained for the same number of epochs.}
  \label{tab:training_inference_comparison}
  \centering
  \scriptsize
  \setlength{\tabcolsep}{3pt}
  \renewcommand{\arraystretch}{1.02}
  \resizebox{0.82\textwidth}{!}{%
  \begin{tabular}{l ccc ccc}
    \toprule
    \multirow{2}{*}{Resolution}
    & \multicolumn{3}{c}{Training time}
    & \multicolumn{3}{c}{Inference} \\
    \cmidrule(lr){2-4}\cmidrule(lr){5-7}
    & Coeff & Basis & CNO
    & Coeff & Basis & CNO \\
    \midrule
    $h=7,\ H=4$ & 1186 & 4409 & 10366 & 0.0036 & 0.0347 & 0.0061 \\
    $h=8,\ H=4$ & 3922 & 12290 & 36076 & 0.0036 & 0.9686 & 0.0064 \\
    $h=8,\ H=5$ & 4508 & 17652 & 36040 & 0.0036 & 1.3884 & 0.0064 \\
    $h=9,\ H=6$ & 27334 & 71862 & 137714 & 0.0086 & 5.5559 & 0.0207 \\
    \bottomrule
  \end{tabular}%
  }
  \vskip -0.1in
\end{table*}

The results show that both LOD-MSNO components combined are faster to train than CNO across all tested resolutions.
At inference time, the neural surrogates are much faster than the corresponding traditional LOD computations, as can be inferred from the per-sample LOD basis and coefficient generation times in Table~\ref{tab:time_comparison}.
This speedup stems from replacing the repeated localized corrector solves and coarse LOD solves with neural-network predictions, and holds even when the classical pipeline is parallelized over 16 CPU cores for the corrector step.
Consequently, the inference-time results indicate that LOD-MSNO is capable of addressing the main computational bottlenecks of the classical LOD method discussed in Section~\ref{sec:lod_limitations}.
The CNO inference time is lower than that of the basis model because CNO directly outputs the solution as an image on the grid.
By contrast, the basis model predicts local basis corrections patch by patch and must assemble the global correction operator and reconstructed solution from these individual local predictions.

\section{Conclusion and outlook}

In this work, we introduced LOD-MSNO, a hybrid neural-operator framework for multiscale Darcy-Flow Problems that combines the problem-adapted approximation structure of the Localized Orthogonal Decomposition method with deep learning-based surrogate modeling. The method is designed based on the Localized Orthogonal Decomposition (LOD) method and aims to replace the expensive numerical computations associated with it through a mainly data-driven approach. In the coefficient-only variant, LOD-MSNO (Coeff) predicts the LOD expansion coefficients and reconstructs the fine-scale solution using the exact corrected multiscale basis, thereby avoiding the assembly and solution of a non-sparse system of equations. In the joint variant, LOD-MSNO (Joint) additionally learns the localized basis corrections, thereby also targeting the second computational bottleneck of the LOD method: the coarse LOD solve and the construction of localized correctors. Through a rigorous error analysis, we gave a precise convergence bound for our models. Numerical experiments also show that using the LOD method as a modeling prior is beneficial for challenging rough and high-contrast permeability fields for the Darcy-flow equation. The comparison with other physics-informed baselines further indicates that explicitly incorporating problem-adapted LOD basis information can be more effective than relying only on PDE residuals for physically consistent solutions.

The following directions remain open for future work. A stronger theoretical analysis of the basis-learning component would be valuable, especially for quantifying how errors in the learned local correction operators propagate to the assembled multiscale basis and the final fine-scale reconstruction. Another important direction is to improve the efficiency of the joint model during inference, for example by reducing patch-wise assembly overhead or by exploiting parallel structure in the local corrector predictions. Lastly, since LOD methods are available for a wide range of problem classes, including elliptic eigenvalue problems and parabolic equations, extending LOD-MSNO to these settings constitutes another promising avenue for future work.

\section*{Acknowledgements}

This work was supported by Samsung Electronics Co., Ltd (IO250307-12279-01), by the Institute of Information \& Communications Technology Planning \& Evaluation (IITP) grant funded by the Korea government (MSIT) [RS-2021-II211341, Artificial Intelligence Graduate School Program (Chung-Ang University)], and by the Chung-Ang University Young Scientist Scholarship in 2025.

\section*{Declaration of generative AI and AI-assisted technologies in the manuscript preparation process}

During the preparation of this work, the author(s) used Prism AI in order to polish and improve the scientific language and readability of human-written texts, as well as to assist with the formatting of formulas and tables. After using this tool/service, the author(s) reviewed and edited the content as needed and take(s) full responsibility for the content of the published article.

\bibliographystyle{elsarticle/elsarticle-harv}
\bibliography{references}

@inproceedings{fanaskov2023spectral,
  title={Spectral neural operators},
  author={Fanaskov, Vladimir Sergeevich and Oseledets, Ivan V},
  booktitle={Doklady Mathematics},
  pages={S226--S232},
  year={2023},
  organization={Springer}
}

@book{bear1988dynamics,
  title={Dynamics of Fluids in Porous Media},
  author={Bear, J.},
  isbn={9780486656755},
  lccn={lc87034940},
  series={Dover Civil and Mechanical Engineering Series},
  url={https://books.google.co.kr/books?id=-XEaxd3hGzoC},
  year={1988},
  publisher={Dover}
}

@article{kalina2023feann,
  title={FEANN: an efficient data-driven multiscale approach based on physics-constrained neural networks and automated data mining},
  author={Karl A. Kalina and Lennart Linden and Jorg Brummund and Markus Kastner},
  journal={Computational Mechanics},
  year={2023},
  volume={71.5},
  pages={827-851},
  url={https://link.springer.com/article/10.1007/s00466-022-02260-0}
}

@Inbook{Yang2024,
author="Yang, Fu-Bao
and Huang, Ji-Ping",
title="Convective Heat Transfer in Porous Materials",
bookTitle="Diffusionics: Diffusion Process Controlled by Diffusion Metamaterials",
year="2024",
publisher="Springer Nature Singapore",
address="Singapore",
pages="129--143",
isbn="978-981-97-0487-3",
doi="10.1007/978-981-97-0487-3_7",
url="https://doi.org/10.1007/978-981-97-0487-3_7"
}

@article{yizheng2026pretraining,
  title={A pretraining-finetuning computational framework for material homogenization},
  author={Wang Yizheng and Xiang Li and Ziming Yan and Shuaifeng Ma and Jinshuai Bai and Bokai Liu and Xiaoying Zhuang and Timoin Rabczuk and Yinghua Liu},
  journal={International Journal of Mechanical Sciences},
  year={2026},
  volume={314},
  pages={111388},
  url={https://www.sciencedirect.com/science/article/abs/pii/S0020740326002444}
}

@article{ricardo2022enhancing,
  title={Enhancing computational fluid dynamics with machine learning},
  author={Vinuesa and Ricardo and and Steven L. Brunton},
  journal={Nature Computational Science},
  year={2022},
  volume={2.6},
  pages={358-366},
  url={https://www.nature.com/articles/s43588-022-00264-7}
}

@article{gentine2018convection,
  title={Could machine learning break the convection parameterization deadlock?},
  author={Pierre Gentine and M. Pritchard and Stephan Rasp and G. Reinaudi and G. Yacalis},
  journal={Geophysical Rresearch Letters},
  year={2018},
  volume={45.11},
  pages={5742-5751},
  url={https://agupubs.onlinelibrary.wiley.com/doi/full/10.1029/2018GL078202}
}

@incollection{KIRKHAM2023431,
title = {Chapter 22 - Electrical analogs for water movement through the soil-plant-atmosphere continuum},
booktitle = {Principles of Soil and Plant Water Relations (Third Edition)},
publisher = {Academic Press},
edition = {Third Edition},
pages = {431-449},
year = {2023},
isbn = {978-0-323-95641-3},
doi = {https://doi.org/10.1016/B978-0-323-95641-3.00030-1},
url = {https://www.sciencedirect.com/science/article/pii/B9780323956413000301},
author = {M.B. Kirkham}
}

@article{peng2021multiscale,
  title={Multiscale modeling meets machine learning: what can we learn?},
  author={Grace C.Y.Peng and Mark Alber and Adrian Buganza Tepole and William R. Cannon and Suvranu De and Savador Dura-Bernal and Krishna Garikipati and George Karniadakis and William W. Lytton and Paris Perdikaris and Linda Petzold and Ellen Kuhl},
  journal={Archives of computational methods in engineering},
  year={2021},
  volume={28.3},
  pages={1017-1037},
  url={https://link.springer.com/article/10.1007/s11831-020-09405-5}
}

@article{ENGWER2019123,
title = {Efficient implementation of the localized orthogonal decomposition method},
journal = {Computer Methods in Applied Mechanics and Engineering},
volume = {350},
pages = {123-153},
year = {2019},
issn = {0045-7825},
doi = {https://doi.org/10.1016/j.cma.2019.02.040},
url = {https://www.sciencedirect.com/science/article/pii/S0045782519301112},
author = {Christian Engwer and Patrick Henning and Axel Målqvist and Daniel Peterseim}
}

@article{RAISSI2019686,
title = {Physics-informed neural networks: A deep learning framework for solving forward and inverse problems involving nonlinear partial differential equations},
journal = {Journal of Computational Physics},
volume = {378},
pages = {686-707},
year = {2019},
issn = {0021-9991},
doi = {https://doi.org/10.1016/j.jcp.2018.10.045},
url = {https://www.sciencedirect.com/science/article/pii/S0021999118307125},
author = {M. Raissi and P. Perdikaris and G.E. Karniadakis}
}

@article{Lu_2021,
   title={Learning nonlinear operators via DeepONet based on the universal approximation theorem of operators},
   volume={3},
   ISSN={2522-5839},
   url={http://dx.doi.org/10.1038/s42256-021-00302-5},
   DOI={10.1038/s42256-021-00302-5},
   number={3},
   journal={Nature Machine Intelligence},
   publisher={Springer Science and Business Media LLC},
   author={Lu, Lu and Jin, Pengzhan and Pang, Guofei and Zhang, Zhongqiang and Karniadakis, George Em},
   year={2021},
   month=mar, pages={218–229} }

@article{Goswami2022PhysicsInformedDN,
  title={Physics-Informed Deep Neural Operator Networks},
  author={Somdatta Goswami and Aniruddha Bora and Yue Yu and George Em Karniadakis},
  journal={ArXiv},
  year={2022},
  volume={abs/2207.05748},
  url={https://api.semanticscholar.org/CorpusID:250626955}
}

@article{SIRIGNANO20181339,
title = {DGM: A deep learning algorithm for solving partial differential equations},
journal = {Journal of Computational Physics},
volume = {375},
pages = {1339-1364},
year = {2018},
issn = {0021-9991},
doi = {https://doi.org/10.1016/j.jcp.2018.08.029},
url = {https://www.sciencedirect.com/science/article/pii/S0021999118305527},
author = {Justin Sirignano and Konstantinos Spiliopoulos}
}

@misc{li2021fourierneuraloperatorparametric,
      title={Fourier Neural Operator for Parametric Partial Differential Equations}, 
      author={Zongyi Li and Nikola Kovachki and Kamyar Azizzadenesheli and Burigede Liu and Kaushik Bhattacharya and Andrew Stuart and Anima Anandkumar},
      year={2021},
      eprint={2010.08895},
      archivePrefix={arXiv},
      primaryClass={cs.LG},
      url={https://arxiv.org/abs/2010.08895}, 
}

@inproceedings{raonic2023convolutional,
author = {Raoni\'{c}, Bogdan and Molinaro, Roberto and De Ryck, Tim and Rohner, Tobias and Bartolucci, Francesca and Alaifari, Rima and Mishra, Siddhartha and de B\'{e}zenac, Emmanuel},
title = {Convolutional neural operators for robust and accurate learning of PDEs},
year = {2023},
publisher = {Curran Associates Inc.},
address = {Red Hook, NY, USA},
booktitle = {Proceedings of the 37th International Conference on Neural Information Processing Systems},
articleno = {3376},
numpages = {14},
location = {New Orleans, LA, USA},
series = {NIPS '23}
}

@article{PINO,
author = {Li, Zongyi and Zheng, Hongkai and Kovachki, Nikola and Jin, David and Chen, Haoxuan and Liu, Burigede and Azizzadenesheli, Kamyar and Anandkumar, Anima},
title = {Physics-Informed Neural Operator for Learning Partial Differential Equations},
year = {2024},
issue_date = {September 2024},
publisher = {Association for Computing Machinery},
address = {New York, NY, USA},
volume = {1},
number = {3},
url = {https://doi.org/10.1145/3648506},
doi = {10.1145/3648506},
journal = {ACM / IMS J. Data Sci.},
month = may,
articleno = {9},
numpages = {27}
}

@book{blanc2023homogenization,
  title={Homogenization Theory for Multiscale Problems: An introduction},
  author={Blanc, Xavier and Le Bris, Claude},
  year={2023},
  publisher={Springer Nature Switzerland},
  series={MS\&A, 21},
  doi={10.1007/978-3-031-21833-0},
  isbn={978-3-031-21832-3}
}

@article{Speconet,
author = {Choi, Junho and Yun, Taehyun and Kim, Namjung and Hong, Youngjoon},
year = {2024},
month = {02},
pages = {116678},
title = {Spectral operator learning for parametric PDEs without data reliance},
volume = {420},
journal = {Computer Methods in Applied Mechanics and Engineering},
doi = {10.1016/j.cma.2023.116678}
}

@article{ulg,
author = {Choi, Junho and Kim, Namjung and Hong, Youngjoon},
year = {2023},
month = {01},
pages = {1-1},
title = {Unsupervised Legendre-Galerkin Neural Network for Solving Partial Differential Equations},
volume = {PP},
journal = {IEEE Access},
doi = {10.1109/ACCESS.2023.3244681}
}

@book{pavliotis2008multiscale,
  title={Multiscale Methods: Averaging and Homogenization},
  author={Pavliotis, Grigorios A and Stuart, Andrew M},
  volume={53},
  year={2008},
  publisher={Springer Science \& Business Media},
  address={New York},
  series={Texts in Applied Mathematics},
  isbn={978-0-387-73828-4}
}

@misc{lee2023hyperdeeponetlearningoperatorcomplex,
      title={HyperDeepONet: learning operator with complex target function space using the limited resources via hypernetwork}, 
      author={Jae Yong Lee and Sung Woong Cho and Hyung Ju Hwang},
      year={2023},
      eprint={2312.15949},
      archivePrefix={arXiv},
      primaryClass={cs.LG},
      url={https://arxiv.org/abs/2312.15949}, 
}

@misc{rahman2022unet,
      title={U-NO: U-shaped Neural Operators}, 
      author={Md Ashiqur Rahman and Zachary E. Ross and Kamyar Azizzadenesheli},
      year={2023},
      eprint={2204.11127},
      archivePrefix={arXiv},
      primaryClass={cs.LG},
      url={https://arxiv.org/abs/2204.11127}, 
}

@article{lodoriginal,
  title={Localization of elliptic multiscale problems},
  author={Målqvist, Axel and Peterseim, Daniel},
  journal={Math. Comput.},
  year={2011},
  volume={83},
  pages={2583-2603},
  url={https://api.semanticscholar.org/CorpusID:267820598}
}

@book{lod_book,
author = {Målqvist, Axel and Peterseim, Daniel},
title = {Numerical Homogenization by Localized Orthogonal Decomposition},
publisher = {Society for Industrial and Applied Mathematics},
year = {2020},
doi = {10.1137/1.9781611976458},
address = {Philadelphia, PA},
edition   = {},
URL = {https://epubs.siam.org/doi/abs/10.1137/1.9781611976458},
eprint = {https://epubs.siam.org/doi/pdf/10.1137/1.9781611976458}
}

@article{lee2025feonet,
author = {Lee, Jae Yong and Ko, Seungchan and Hong, Youngjoon},
title = {Finite Element Operator Network for Solving Elliptic-Type Parametric PDEs},
journal = {SIAM Journal on Scientific Computing},
volume = {47},
number = {2},
pages = {C501-C528},
year = {2025},
doi = {10.1137/23M1623707},
URL = {https://doi.org/10.1137/23M1623707},
eprint = { https://doi.org/10.1137/23M1623707}
}

@article{parab,
author = {Målqvist, Axel and Persson, Anna},
year = {2018},
month = {01},
pages = {},
title = {Multiscale techniques for parabolic equations},
volume = {138},
journal = {Numerische Mathematik},
doi = {10.1007/s00211-017-0905-7}
}

@inproceedings{Transolver,
author = {Wu, Haixu and Luo, Huakun and Wang, Haowen and Wang, Jianmin and Long, Mingsheng},
title = {Transolver: a fast transformer solver for PDEs on general geometries},
year = {2024},
publisher = {JMLR.org},
booktitle = {Proceedings of the 41st International Conference on Machine Learning},
articleno = {2200},
numpages = {25},
location = {Vienna, Austria},
series = {ICML'24}
}

@article{fair,
title = {A comprehensive and fair comparison of two neural operators (with practical extensions) based on FAIR data},
journal = {Computer Methods in Applied Mechanics and Engineering},
volume = {393},
pages = {114778},
year = {2022},
issn = {0045-7825},
doi = {https://doi.org/10.1016/j.cma.2022.114778},
url = {https://www.sciencedirect.com/science/article/pii/S0045782522001207},
author = {Lu Lu and Xuhui Meng and Shengze Cai and Zhiping Mao and Somdatta Goswami and Zhongqiang Zhang and George Em Karniadakis}
}

@article {MR3283907,
    AUTHOR = {Bonizzoni, Francesca and Nobile, Fabio},
     TITLE = {Perturbation analysis for the {D}arcy problem with log-normal
              permeability},
   JOURNAL = {SIAM/ASA J. Uncertain. Quantif.},
  FJOURNAL = {SIAM/ASA Journal on Uncertainty Quantification},
    VOLUME = {2},
      YEAR = {2014},
    NUMBER = {1},
     PAGES = {223--244},
      ISSN = {2166-2525},
   MRCLASS = {35R60 (35B20 35J15 41A58 60H25 76S05)},
  MRNUMBER = {3283907},
MRREVIEWER = {S\"oren\ Dobbersch\"utz},
       DOI = {10.1137/130949415},
       URL = {https://doi.org/10.1137/130949415},
}

@inproceedings{10.5555/3327546.3327630,
author = {Liu, Rosanne and Lehman, Joel and Molino, Piero and Such, Felipe Petroski and Frank, Eric and Sergeev, Alex and Yosinski, Jason},
title = {An intriguing failing of convolutional neural networks and the CoordConv solution},
year = {2018},
publisher = {Curran Associates Inc.},
address = {Red Hook, NY, USA},
booktitle = {Proceedings of the 32nd International Conference on Neural Information Processing Systems},
pages = {9628–9639},
numpages = {12},
location = {Montr\'{e}al, Canada},
series = {NIPS'18}
}

@InProceedings{Wu_2018_ECCV,
author = {Wu, Yuxin and He, Kaiming},
title = {Group Normalization},
booktitle = {Proceedings of the European Conference on Computer Vision (ECCV)},
month = {September},
year = {2018}
}

@article{lu2021learning,
author = {Lu, Lu and Jin, Pengzhan and Pang, Guofei and Zhang, Zhongqiang and Karniadakis, George},
year = {2021},
month = {03},
pages = {218-229},
title = {Learning nonlinear operators via DeepONet based on the universal approximation theorem of operators},
volume = {3},
journal = {Nature Machine Intelligence},
doi = {10.1038/s42256-021-00302-5}
}

@book{efendiev2009multiscale,
  title={Multiscale finite element methods: theory and applications},
  author={Efendiev, Yalchin and Hou, Thomas Y},
  year={2009},
  publisher={Springer Science \& Business Media}
}

@article{hou1997multiscale,
  title={A multiscale finite element method for elliptic problems in composite materials and porous media},
  author={Hou, Thomas Y and Wu, Xiao-Hui},
  journal={Journal of computational physics},
  volume={134},
  number={1},
  pages={169--189},
  year={1997},
  publisher={Elsevier}
}

@InProceedings{cao2021choose,
  author    = {Cao, Shuhao},
  title     = {Choose a {T}ransformer: {F}ourier or {G}alerkin},
  booktitle = {Advances in Neural Information Processing Systems},
  year      = {2021},
}

@InProceedings{wang2025cvit,
  author    = {Wang, Sifan and Seidman, Jacob H. and Sankaran, Shyam and Wang, Hanwen and Pappas, George J. and Perdikaris, Paris},
  title     = {{CV}i{T}: {C}ontinuous {V}ision {T}ransformer for {O}perator {L}earning},
  booktitle = {The Thirteenth International Conference on Learning Representations},
  year      = {2025},
}

@InProceedings{pmlr-v235-hao24d,
  author    = {Hao, Zhongkai and Su, Chang and Liu, Songming and Berner, Julius and Ying, Chengyang and Su, Hang and Anandkumar, Anima and Song, Jian and Zhu, Jun},
  title     = {{DPOT}: Auto-Regressive Denoising Operator Transformer for Large-Scale {PDE} Pre-Training},
  booktitle = {Proceedings of the 41st International Conference on Machine Learning},
  series    = {Proceedings of Machine Learning Research},
  volume    = {235},
  publisher = {PMLR},
  year      = {2024},
  pages     = {17616--17635},
}

\newpage
\appendix

\section{Theoretical Analysis}

\subsection{SPD-property of the LOD-stiffness matrix}

\begin{AppendixLemma}
Let $D \subset \mathbb{R}^d$ be the physical domain and let $\Omega$ be the input domain of the neural network.
Assume that $A(x,\omega)\in\mathbb{R}^{d\times d}$ is symmetric and uniformly elliptic, i.e.,
\begin{equation*}
    \alpha |\xi|^2 \le \xi^T A(x,\omega)\xi \le \beta |\xi|^2, \quad
\forall \xi \in \mathbb{R}^d,
\end{equation*}
for a.e.\ $x\in D$ and a.e.\ $\omega\in\Omega$. For each fixed $\omega\in\Omega$, define the bilinear form
\begin{equation*}
a_\omega(u,v)
:=
\int_D (A(x,\omega)\nabla u)\cdot \nabla v\,dx,
\end{equation*}
and let $A_H(\omega)$ be the coarse-scale stiffness matrix associated with $a_\omega$. Let $V_{\mathrm{ms}}(\omega)$ be the matrix that represents the change of basis from the nodal basis functions of $V_H$ to the LOD-corrected basis, and define 
\begin{equation*}
A_{\mathrm{LOD}}(\omega)
:=
V_{\mathrm{ms}}(\omega)\,A_h(\omega)\,V_{\mathrm{ms}}(\omega)^T.
\end{equation*}
Then, for a.e.\ $\omega\in\Omega$, the matrix $A_{\mathrm{LOD}}(\omega)$ is symmetric positive definite.
\end{AppendixLemma}

\begin{proof}
We first establish that the bilinear form $a_\omega$ is coercive on $V=H^1(D)$. Fix $\omega\in\Omega$ and let $u \in H^1_0(D)$. The Poincaré inequality gives:
\begin{equation*}
    {|| u ||}_{L^2(D)} \leq C_p {|| \nabla u ||}_{L^2(D)}
\end{equation*}
Therefore:
\begin{align*}
\|u\|_{H^1(D)}^2
&= \|u\|_{L^2(D)}^2 + \|\nabla u\|_{L^2(D)}^2 \\
&\le (C_p^2+1)\|\nabla u\|_{L^2(D)}^2.
\end{align*}
Using this and the uniform ellipticity of $A$ we obtain
\begin{align*}
a_\omega(u,u)
&= \int_D (A(x,\omega)\nabla u)\cdot \nabla u\,dx \\
&\ge \alpha\|\nabla u\|_{L^2(D)}^2
\ge \frac{\alpha}{C_p^2+1}\|u\|_{H^1(D)}^2.
\end{align*}
Next, the coarse-scale stiffness matrix $A_h(\omega)$ associated with $a_{\omega}$, defined by 
\begin{equation*}
   {(A_H(\omega))}_{ij} = a_{\omega}(\Phi_i, \Phi_j)
\end{equation*}
is positive-definite: Let $c\neq 0$ and write $v_H \in V_H$ as $v_h=\sum_i c_i \varphi_i$ in the nodal basis.
Then $v_H\neq 0$, and therefore
\begin{equation*}
    c^T A_h(\omega)c
=
a_\omega(v_h,v_h)
>
0
\end{equation*}
The LOD construction gives the multiscale space
\[
V_{\mathrm{ms}}(\omega)=(1-\mathcal C^{\ell}(\omega))V_H,
\]
and the map $(1-\mathcal C^{\ell}(\omega)):V_H\to V_{\mathrm{ms}}(\omega)$ is bijective. Therefore, the associated change-of-basis matrix $B_{\mathrm{ms}}(\omega)$ has full rank. Let $x\neq 0$ and set
\[
y := B_{\mathrm{ms}}(\omega)^T x.
\]
Since $B_{\mathrm{ms}}(\omega)$ has full rank, we have $y\neq 0$.
Thus
\[
x^T A_{\mathrm{LOD}}(\omega)x
=
x^T B_{\mathrm{ms}}(\omega)A_H(\omega)B_{\mathrm{ms}}(\omega)^T x
=
y^T A_H(\omega)y
>
0.
\]
Symmetry follows from the symmetry of $A_h(\omega)$.
\end{proof}

\subsection{Proof of Theorem~\ref{thm:energy_approx} - Approximation error for the energy-based loss function}

\noindent To prove Theorem~\ref{thm:energy_approx}, we shall use the following Lemma:

\begin{AppendixLemma}[Uniform norm equivalence]
Denote with $A_{\mathrm{LOD}}(\omega)$ the LOD-stiffness matrix and let
$\lambda_{\min}(\omega)$ and $\lambda_{\max}(\omega)$ denote its smallest and largest eigenvalues.
Then for every measurable $e:\Omega\to\mathbb{R}^{N_H}$ we have,
\begin{equation*}
    \sqrt{\lambda_{\min}(\omega)}\,|e(\omega)|
    \le
    |A_{\mathrm{LOD}}(\omega)^{1/2}e(\omega)|
    \le
    \sqrt{\lambda_{\max}(\omega)}\,|e(\omega)|
\end{equation*}
for a.e.\ $\omega\in\Omega$.
Consequently,
\[
\|A_{\mathrm{LOD}}^{1/2} e\|_{L^\infty(\Omega)}
\le
\|\sqrt{\lambda_{\max}}\|_{L^\infty(\Omega)}
\|e\|_{L^\infty(\Omega)},
\]
and
\[
\|e\|_{L^\infty(\Omega)}
\le
\left\|\frac{1}{\sqrt{\lambda_{\min}}}\right\|_{L^\infty(\Omega)}
\|A_{\mathrm{LOD}}^{1/2} e\|_{L^\infty(\Omega)}.
\]
\end{AppendixLemma}

\begin{proof}
Since $A_{\mathrm{LOD}}(\omega)$ is symmetric positive definite, it admits an orthonormal eigenbasis.
Hence, for any $z\in\mathbb{R}^{N_H}$,
\[
\lambda_{\min}(\omega)|z|^2
\le
z^T A_{\mathrm{LOD}}(\omega) z
=
|A_{\mathrm{LOD}}(\omega)^{1/2}z|^2
\le
\lambda_{\max}(\omega)|z|^2.
\]
Taking $z=e(\omega)$ and then square roots yields
\[
\sqrt{\lambda_{\min}(\omega)}\,|e(\omega)|
\le
|A_{\mathrm{LOD}}(\omega)^{1/2}e(\omega)|
\le
\sqrt{\lambda_{\max}(\omega)}\,|e(\omega)|.
\]
Taking essential suprema over $\omega\in\Omega$ gives the two $L^\infty$ estimates.
\end{proof}

\begin{AppendixTheorem}[Approximation error for the energy-based loss]
\label{thm:energy_approx_app}

Denote by $\lambda_{\min}(\omega)$ and $\lambda_{\max}(\omega)$ the smallest and biggest eigenvalues of $\boldsymbol{A}_{\mathrm{LOD}}$ respectively. Assume that there exists a positive constant $c_0 > 0$ such that
\begin{equation*}
\operatorname*{ess\,inf}_{\omega\in\Omega}\lambda_{\min}(\omega)\ge c_0>0.
\end{equation*}
Denote by $\boldsymbol{u}_{\mathrm{LOD}} \colon \Omega \longrightarrow \mathbb{R}^{N_H}$ the true coefficient map of the LOD approximation, i.e.\ the solution of the linear system~\eqref{eq:lod_linear_eq} for a given $\omega \in \Omega$
%, and use the population energy loss $\mathcal{L}_{E}$ from~\eqref{eq:energy_loss_population}.
Let $\mathcal N$ be any neural network class, and define
\[
{\boldsymbol{u}}_{\mathrm{LOD}}^E
\in
\operatorname*{argmin}_{\boldsymbol{u}\in\mathcal N}\mathcal L_E(\boldsymbol{u}).
\]
Then
\[
\|{\boldsymbol{u}}_{\mathrm{LOD}}-{\boldsymbol{u}}_{\mathrm{LOD}}^E\|_{L^1(\Omega)}
\le
c_0^{-1/2}
\inf_{\boldsymbol{u}\in\mathcal N}
\|A_{\mathrm{LOD}}^{1/2}(\boldsymbol{u}-\boldsymbol{u}_{\mathrm{LOD}})\|_{L^1(\Omega)}.
\]
If, in addition,
\[
\|\sqrt{\lambda_{\max}}\|_{L^\infty(\Omega)} := c_1 <\infty,
\]
for some $c_1>0$, then
\[
\|{\boldsymbol{u}}_{\mathrm{LOD}}-{\boldsymbol{u}}_{\mathrm{LOD}}^E\|_{L^1(\Omega)}
\le
\frac{c_1}{\sqrt{c_0}}
\inf_{\boldsymbol{u}\in\mathcal N}
\|(\boldsymbol{u}-\boldsymbol{u}_{\mathrm{LOD}})\|_{L^1(\Omega)}
\]
\end{AppendixTheorem}

\begin{proof}
Set
\[
e(\omega):={\boldsymbol{u}}_{\mathrm{LOD}} (\omega)-{\boldsymbol{u}}_{\mathrm{LOD}}^E(\omega).
\]
By the uniform norm equivalence lemma,
\[
|e(\omega)|
\le
\frac{|A^{1/2}_{\mathrm{LOD}}(\omega)e(\omega)|}{\sqrt{\lambda_{\min}(\omega)}}
\]
for a.e.\ $\omega\in\Omega$.
Integrating over $\Omega$ gives
\begin{equation*}
    \|{\boldsymbol{u}}_{\mathrm{LOD}}-{\boldsymbol{u}}_{\mathrm{LOD}}^E\|_{L^1(\Omega)}
\le
\left\|\frac{1}{\sqrt{\lambda_{\min}}}\right\|_{L^\infty(\Omega)}
\|A_{\mathrm{LOD}}^{1/2}({\boldsymbol{u}}_{\mathrm{LOD}}-\boldsymbol{u}_{\mathrm{LOD}}^E)\|_{L^1(\Omega)}.
\end{equation*}
Since ${\boldsymbol{u}}_{\mathrm{LOD}}^E$ minimizes $\mathcal L_E$ over $\mathcal N$,
\[
\|A_{\mathrm{LOD}}^{1/2}(\boldsymbol{u}-\boldsymbol{u}_{\mathrm{LOD}})\|_{L^1(\Omega)}
= \inf_{\boldsymbol{u} \in\mathcal N}\mathcal L_E(\boldsymbol{u}).
\]
Combining the two estimates yields
\[
\|{\boldsymbol{u}}_{\mathrm{LOD}}-{\boldsymbol{u}}_{\mathrm{LOD}}^E\|_{L^1(\Omega)}
\le
\left\|\frac{1}{\sqrt{\lambda_{\min}}}\right\|_{L^\infty(\Omega)}
\inf_{\boldsymbol{u}\in\mathcal N}
\|A_{\mathrm{LOD}}^{1/2}(\boldsymbol{u}-\boldsymbol{u}_{\mathrm{LOD}})\|_{L^1(\Omega)}.
\]
The bound with $c_0^{-1/2}$ follows from
\[
\left\|\frac{1}{\sqrt{\lambda_{\min}}}\right\|_{L^\infty(\Omega)}
\le
c_0^{-1/2}.
\]
For the final estimate, use the upper spectral bound
\[
|A_{\mathrm{LOD}}(\omega)^{1/2}v|
\le
\sqrt{\lambda_{\max}(\omega)}\,|v|
\]
with $v={\boldsymbol{u}}_{\mathrm{LOD}} (\omega)-{\boldsymbol{u}}_{\mathrm{LOD}}^E(\omega)$ to get
\[
\|A_{\mathrm{LOD}}^{1/2}(\boldsymbol{u}-\boldsymbol{u}_{\mathrm{LOD}})\|_{L^1(\Omega)}
\le
\|\sqrt{\lambda_{\max}}\|_{L^\infty(\Omega)}
\|(\boldsymbol{u}-\boldsymbol{u}_{\mathrm{LOD}})\|_{L^1(\Omega)}.
\]
Substituting this into the previous bound gives
\[
\|{\boldsymbol{u}}_{\mathrm{LOD}}-{\boldsymbol{u}}_{\mathrm{LOD}}^E\|_{L^1(\Omega)}
\le
\left\|\frac{1}{\sqrt{\lambda_{\min}}}\right\|_{L^\infty(\Omega)}
\|\sqrt{\lambda_{\max}}\|_{L^\infty(\Omega)}
\inf_{\boldsymbol{u}\in\mathcal N}
\|(\boldsymbol{u}-\boldsymbol{u}_{\mathrm{LOD}})\|_{L^1(\Omega)}.
\]
\end{proof}

\subsection{Proof of Theorem~\ref{thm:energy_gen} - Generalization error for the energy-based loss function}

\begin{AppendixTheorem}[Generalization error for the energy-based loss]
\label{thm:energy_gen_app}
Assume that
\[
\operatorname*{ess\,inf}_{\omega\in\Omega}\lambda_{\min}(\omega)\ge c_0>0,
\qquad
\operatorname*{ess\,sup}_{\omega\in\Omega}\lambda_{\max}(\omega)\leq c_1 <\infty.
\]
Let $\omega_1,\dots,\omega_M$ be i.i.d.\ samples in $\Omega$.
Define the empirical energy loss as the Monte-Carlo approximation of the energy loss
\[
\mathcal L_{E,M}(\boldsymbol{u})
:=
\frac{|\Omega|}{M}\sum_{i=1}^M
\left|
\boldsymbol{A}_{\mathrm{LOD}}(\omega_i)^{1/2}
\bigl({\boldsymbol{u}}_{\mathrm{LOD}} (\omega_i)-\boldsymbol{u}(\omega_i)\bigr)
\right|.
\]
and let
\[
\widehat{\boldsymbol{u}}^{E,M}_{\mathrm{LOD}}
\in
\operatorname*{argmin}_{\boldsymbol{u}\in\mathcal N}\mathcal L_{E,M}(\boldsymbol{u})
\]
be a minimizer of the empirical energy loss function. Define the function class associated with the Neural Network class $\mathcal{N}$ and the energy loss
\[
\mathcal G^{E}
:=
\Bigl\{
\omega\mapsto
\left|
A^{1/2}_{\mathrm{LOD}}(\omega)
\bigl({\boldsymbol{u}}_{\mathrm{LOD}} (\omega)-\boldsymbol{u}(\omega)\bigr)
\right|
:\ \boldsymbol{u} \in \mathcal N
\Bigr\}.
\]
Then
\begin{equation*}
\begin{aligned}
&\mathbb E\!\left[
\left\|
\widehat{\boldsymbol{u}}^{E}_{\mathrm{LOD}}
-\widehat{\boldsymbol{u}}^{E,M}_{\mathrm{LOD}}
\right\|_{L^1(\Omega)}
\right] \\
&\lesssim
\left\|
\frac{1}{\sqrt{\lambda_{\min}}}
\right\|_{L^\infty(\Omega)}
R_M(\mathcal G^{E}) \\
&\quad+
\left\|\frac{1}{\sqrt{\lambda_{\min}}}\right\|_{L^\infty(\Omega)}
\|\sqrt{\lambda_{\max}}\|_{L^\infty(\Omega)}
\inf_{\boldsymbol{u}\in\mathcal N}
\|(\boldsymbol{u}-\boldsymbol{u}_{\mathrm{LOD}})\|_{L^1(\Omega)}.
\end{aligned}
\end{equation*}
Here, $R_M(\mathcal G^{E})$ denotes the empirical Rademacher complexity of the function class $\mathcal{G}^{E}$ and the expectation is taken with respect to the random training samples
$\omega_1,\dots,\omega_M$.
\end{AppendixTheorem}

\begin{proof}
Set
\[
C_{\min}
:=
\left\|
\frac{1}{\sqrt{\lambda_{\min}}}
\right\|_{L^\infty(\Omega)}.
\]
By the lower spectral bound,
\[
|v(\omega)|
\le
\frac{|A^{1/2}_{\mathrm{LOD}}(\omega)v(\omega)|}{\sqrt{\lambda_{\min}(\omega)}}
\]
for a.e.\ $\omega\in\Omega$. Hence

\begin{align*}
&\mathbb E\!\left[
\left\|
\widehat{\boldsymbol{u}}^{E}_{\mathrm{LOD}}-\widehat{\boldsymbol{u}}^{E,M}_{\mathrm{LOD}}
\right\|_{L^1(\Omega)}
\right] \\
&\le
C_{\min}\,
\mathbb E\!\left[
\left\|
A_{\mathrm{LOD}}^{1/2}
\bigl(
\widehat{\boldsymbol{u}}^{E}_{\mathrm{LOD}}-\widehat{\boldsymbol{u}}^{E,M}_{\mathrm{LOD}}
\bigr)
\right\|_{L^1(\Omega)}
\right] \\
&\le
C_{\min}\,\mathbb E\!\left[
\left\|
A_{\mathrm{LOD}}^{1/2}
(\widehat{\boldsymbol{u}}^{E}_{\mathrm{LOD}}-{\boldsymbol{u}}_{\mathrm{LOD}})
\right\|_{L^1(\Omega)}
\right] \\
&\quad+
C_{\min}\,\mathbb E\!\left[
\left\|
A_{\mathrm{LOD}}^{1/2}
(\widehat{\boldsymbol{u}}^{E,M}_{\mathrm{LOD}}-{\boldsymbol{u}}_{\mathrm{LOD}})
\right\|_{L^1(\Omega)}
\right] \\
&=
C_{\min}\,\mathbb E\!\left[
\mathcal L_E(\widehat{\boldsymbol{u}}^{E}_{\mathrm{LOD}})
+
\mathcal L_E(\widehat{\boldsymbol{u}}^{E,M}_{\mathrm{LOD}})
\right].
\end{align*}
Since $\widehat{\boldsymbol{u}}^{E}_{\mathrm{LOD}}$ minimizes the population loss,
\[
\mathcal L_E(\widehat{\boldsymbol{u}}^{E}_{\mathrm{LOD}})
\le
\mathcal L_E(\widehat{\boldsymbol{u}}^{E,M}_{\mathrm{LOD}}),
\]
and therefore
\[
\mathbb E\!\left[
\left\|
\widehat{\boldsymbol{u}}^{E}_{\mathrm{LOD}}-\widehat{\boldsymbol{u}}^{E,M}_{\mathrm{LOD}}
\right\|_{L^1(\Omega)}
\right]
\le
2C_{\min}\,
\mathbb E\!\left[
\mathcal L_E(\widehat{\boldsymbol{u}}^{E,M}_{\mathrm{LOD}})
\right].
\]
Now,
\begingroup\footnotesize
\begin{align*}
\mathcal L_E(\widehat{\boldsymbol{u}}^{E,M}_{\mathrm{LOD}})
&=
\Bigl(
\mathcal L_{E}(\widehat{\boldsymbol{u}}^{E,M}_{\mathrm{LOD}})
-
\mathcal L_{E,M}(\widehat{\boldsymbol{u}}^{E,M}_{\mathrm{LOD}})
\Bigr)
+
\mathcal L_{E,M}(\widehat{\boldsymbol{u}}^{E,M}_{\mathrm{LOD}}) \\
&\le
\sup_{\boldsymbol{u} \in \mathcal N}
\bigl|
\mathcal L_E(\boldsymbol{u})-\mathcal L_{E,M}(\boldsymbol{u})
\bigr|
+
\mathcal L_{E,M}(\widehat{\boldsymbol{u}}^{E,M}_{\mathrm{LOD}}) \\
&\le
\sup_{\boldsymbol{u} \in \mathcal N}
\bigl|
\mathcal L_E(\boldsymbol{u})-\mathcal L_{E,M}(\boldsymbol{u})
\bigr|
+
\mathcal L_{E,M}(\widehat{\boldsymbol{u}}^{E}_{\mathrm{LOD}}) \\
&=
\sup_{\boldsymbol{u} \in \mathcal N}
\bigl|
\mathcal L_E(\boldsymbol{u})-\mathcal L_{E,M}(\boldsymbol{u})
\bigr|
+
\Bigl(
\mathcal L_{E,M}(\widehat{\boldsymbol{u}}^{E}_{\mathrm{LOD}})
-
\mathcal L_E(\widehat{\boldsymbol{u}}^{E}_{\mathrm{LOD}})
\Bigr)
+
\mathcal L_E(\widehat{\boldsymbol{u}}^{E}_{\mathrm{LOD}}) \\
&\le
2\sup_{\boldsymbol{u} \in \mathcal N}
\bigl|
\mathcal L_E(\boldsymbol{u})-\mathcal L_{E,M}(\boldsymbol{u})
\bigr| \\
&\quad+
\inf_{\boldsymbol{u} \in \mathcal N}\mathcal L_E(\boldsymbol{u}).
\end{align*}
\endgroup
Taking expectation and using the fact that the Rademacher complexity serves as a measure between the error of the population risk minimizer and the empirical risk minimizer, we obtain:
\[
\mathbb E\!\left[
\sup_{\boldsymbol{u} \in \mathcal N}
\bigl|
\mathcal L_E(\boldsymbol{u})-\mathcal L_{E,M}(\boldsymbol{u})
\bigr|
\right]
\lesssim
R_M(\mathcal G^{E}),
\]
we obtain
\begin{equation*}
\
\mathbb E\!\left[
\mathcal L_E(\widehat{\boldsymbol{u}}^{E,M}_{\mathrm{LOD}})
\right] \lesssim R_M(\mathcal G^{E}) +
\inf_{\boldsymbol{u} \in \mathcal N}\mathcal L_E(\boldsymbol{u}).
\end{equation*}
Substituting this into the previous estimate yields
\[
\mathbb E\!\left[
\left\|
\widehat{\boldsymbol{u}}^{E}_{\mathrm{LOD}}-\widehat{\boldsymbol{u}}^{E,M}_{\mathrm{LOD}}
\right\|_{L^1(\Omega)}
\right]
\lesssim
C_{\min}
\left( R_M(\mathcal G^{E})
+
\inf_{\boldsymbol{u} \in \mathcal N}\mathcal L_E(\boldsymbol{u})
\right)
.
\]
Finally, by the upper spectral bound,
\[
\mathcal L_E(\boldsymbol{u})
=
\|A_{\mathrm{LOD}}^{1/2}(\boldsymbol{u}-\boldsymbol{u}_{\mathrm{LOD}})\|_{L^1(\Omega)}
\le
\|\sqrt{\lambda_{\max}}\|_{L^\infty(\Omega)}
\|(\boldsymbol{u}-\boldsymbol{u}_{\mathrm{LOD}})\|_{L^1(\Omega)}
\]
Therefore,
\[
\inf_{\boldsymbol{u} \in \mathcal N} \mathcal L_E(\boldsymbol{u})
\le
\|\sqrt{\lambda_{\max}}\|_{L^\infty(\Omega)}
\inf_{\boldsymbol{u} \in \mathcal N}
\|(\boldsymbol{u}-\boldsymbol{u}_{\mathrm{LOD}})\|_{L^1(\Omega)}.
\]
Combining the above bounds gives the claimed conclusion,
\begin{align*}
&\mathbb E\!\left[
\left\|
\widehat{\boldsymbol{u}}^{E}_{\mathrm{LOD}}
-\widehat{\boldsymbol{u}}^{E,M}_{\mathrm{LOD}}
\right\|_{L^1(\Omega)}
\right] \\
&\lesssim
\left\|
\frac{1}{\sqrt{\lambda_{\min}}}
\right\|_{L^\infty(\Omega)}
R_M(\mathcal G^{E}) \\
&\quad+
\left\|\frac{1}{\sqrt{\lambda_{\min}}}\right\|_{L^\infty(\Omega)}
\|\sqrt{\lambda_{\max}}\|_{L^\infty(\Omega)}
\inf_{\boldsymbol{u}\in\mathcal N}
\|(\boldsymbol{u}-\boldsymbol{u}_{\mathrm{LOD}})\|_{L^1(\Omega)}.
\qedhere
\end{align*}
\end{proof}

\section{Baseline models}

\subsection{Common input--output format and training losses}
\label{app:baseline_common}

All baselines learn the solution operator for (\ref{darcy_eq}) on a regular $s\times s$ grid. The input is a coefficient field (and optional engineered channels), and the output is the scalar solution field.
Concretely, the network input is a tensor in $\mathbb{R}^{B\times s\times s\times C_{\mathrm{in}}}$ (or equivalently $B\times C_{\mathrm{in}}\times s\times s$),
and the output prediction is $u_\theta \in \mathbb{R}^{B\times s\times s}$.
Here, $C_{\mathrm{in}}$ is the number of input channels used in the experiment.
All baselines except FEONet are trained on fine-mesh finite element reference solutions $u_h$ defined by~\eqref{eq:fine_fem}, i.e $s=\frac{1}{h}$. FEONet is treated separately because it is trained through its finite-element weak-form residual rather than by direct supervision on these fine-mesh solution labels.

% For supervised training, we use the standard data loss
% $\mathcal{L}_{\text{data}}(\theta)=\|u_\theta-u\|_2^2$.
% For physics-informed training of PI-DeepONet, we additionally enforce the Darcy equation $-\nabla \cdot \big(\kappa(x)\nabla u(x)\big)=f(x)$ via a finite-volume/finite-difference style residual on the interior grid points. Let $h = 1/(s-1)$ and let $u_{ij}$ and $a_{ij}$ be the solution and coefficient values on the grid. We define arithmetic averages on half-indices
% \begin{equation}
% a_{i-\frac12,j}=\tfrac12(a_{ij}+a_{i-1,j}),\quad
% a_{i+\frac12,j}=\tfrac12(a_{ij}+a_{i+1,j}),\quad
% a_{i,j-\frac12}=\tfrac12(a_{ij}+a_{i,j-1}),\quad
% a_{i,j+\frac12}=\tfrac12(a_{ij}+a_{i,j+1}),
% \end{equation}
% and the interior residual (matching our implementation) as
% \begin{equation}
% r_{ij}
% =
% \big(a_{i-\frac12,j}+a_{i+\frac12,j}+a_{i,j-\frac12}+a_{i,j+\frac12}\big)u_{ij}
% -a_{i-\frac12,j}u_{i-1,j}-a_{i+\frac12,j}u_{i+1,j}
% -a_{i,j-\frac12}u_{i,j-1}-a_{i,j+\frac12}u_{i,j+1}
% -f h^2.
% \end{equation}

% Our PI objective is $\mathcal{L}_{\text{PI}}(\theta) = \mathcal{L}_{\text{data}}(\theta) + \lambda \mathcal{L}_{\text{pde}}(\theta)$ where $\mathcal{L}_{\text{pde}}(\theta) := \mathbb{E}\left[ \|r\|_2^2 \right]$

For supervised training, we use the standard data loss
\[
\mathcal{L}_{\mathrm{data}}(\theta)=\|u_\theta-u\|_2^2.
\]
For physics-informed training of PI-DeepONet and PINO, we additionally enforce the Darcy equation
\[
-\nabla \cdot \big(\kappa(x)\nabla u(x)\big)=f(x)
\]
through coordinate-based automatic differentiation rather than a finite-volume/finite-difference stencil. Let
\[
\mathcal{G}=\{x_{ij}\}_{i,j=1}^{s}
\]
denote the regular grid points, and let $u_\theta(x;\kappa)$ be the network prediction evaluated at coordinate $x\in\Omega$. We define the pointwise residual by
\[
r_\theta(x;\kappa)
=
-\partial_{x_1}\!\left(\kappa(x)\,\partial_{x_1}u_\theta(x;\kappa)\right)
-\partial_{x_2}\!\left(\kappa(x)\,\partial_{x_2}u_\theta(x;\kappa)\right)
-f(x).
\]
In the implementation, the spatial coordinates are treated as differentiable inputs, $\nabla u_\theta$ is obtained by automatic differentiation with respect to those coordinates, and $\kappa(x)$ is evaluated from the raw coefficient field at the same coordinates by differentiable interpolation. The physics-informed loss functions are then
\[
\mathcal{L}_{\mathrm{pde}}(\theta)
=
\frac{1}{|\mathcal{G}|}\sum_{x\in\mathcal{G}} |r_\theta(x;\kappa)|^2,
\qquad
\mathcal{L}_{\mathrm{PI}}(\theta)
=
\mathcal{L}_{\mathrm{data}}(\theta)+\lambda \mathcal{L}_{\mathrm{pde}}(\theta).
\]

% ----------------------------
% Hyperparameter macros (edit numbers here once)
% ----------------------------
\newcommand{\FNOlayers}{\textcolor{black}{4}}
\newcommand{\FNOwidth}{\textcolor{black}{64}}
\newcommand{\FNOmodes}{\textcolor{black}{12}}

\newcommand{\CNOlevels}{\textcolor{black}{4}}
\newcommand{\CNOchannels}{\textcolor{black}{64}}
\newcommand{\CNOrblocks}{\textcolor{black}{2}}

\newcommand{\UNOdepth}{\textcolor{black}{4}}
\newcommand{\UNOwidth}{\textcolor{black}{64}}

\newcommand{\DONp}{\textcolor{black}{128}}
\newcommand{\DONbranchDepth}{\textcolor{black}{4}}
\newcommand{\DONtrunkDepth}{\textcolor{black}{4}}
\newcommand{\DONwidth}{\textcolor{black}{128}}

\newcommand{\TSdim}{\textcolor{black}{128}}
\newcommand{\TSslices}{\textcolor{black}{64}}
\newcommand{\TSheads}{\textcolor{black}{8}}
\newcommand{\TStau}{\textcolor{black}{0.5}}

\subsection{Fourier neural operator (FNO)}
\label{app:fno}

The Fourier Neural Operator (FNO), introduced by Li et al. \citep{li2021fourierneuraloperatorparametric}, is a neural operator architecture for learning mappings between infinite-dimensional function spaces, with a particular focus on approximating solution operators of parametric partial differential equations. The key idea of the FNO is to parameterize an integral operator in Fourier space, enabling efficient global convolution and resolution-invariant generalization.

Given a function $v_t$ defined on a spatial domain, the FNO constructs a learnable integral operator $\mathcal{K}(\phi)$ of the form
\begin{equation*}
      \left( \mathcal{K}(\phi)v_t \right)(x)
      =
      \mathcal{F}^{-1} \left( \mathcal{F}(\kappa_{\phi}) \cdot (\mathcal{F}v_t) \right)(x),
\end{equation*}
where $\mathcal{F}$ and $\mathcal{F}^{-1}$ denote the Fourier transform and its inverse, respectively, and $\kappa_{\phi}$ is a learnable kernel. This formulation corresponds to a global convolution in physical space and allows the operator to capture long-range dependencies efficiently.

In practice, the continuous Fourier transforms are replaced by discrete fast Fourier transforms. Specifically:
\begin{itemize}
    \item The operators $\mathcal{F}$ and $\mathcal{F}^{-1}$ are replaced by the discrete fast Fourier transforms $\hat{\mathcal{F}}$ and $\hat{\mathcal{F}}^{-1}$.
    \item The input $\hat{v} \in \mathbb{R}^{s_1 \times \dots \times s_d \times n}$ represents the evaluation of the function $v$ on a uniform grid with $s_j \in \mathbb{N}$ points in each spatial dimension.
    \item The Fourier transform of the kernel $\mathcal{F}(\kappa)$ is replaced by a complex-valued weight tensor
    $\boldsymbol{T} \in \mathbb{C}^{s_1 \times \dots \times s_d \times n \times n}$,
    whose entries correspond to the learnable Fourier modes of the kernel.
\end{itemize}

To reduce computational complexity and improve generalization, the FNO typically employs a low-frequency truncation in Fourier space. Concretely, only a prescribed number of low-frequency modes are retained in each spatial dimension, while the remaining high-frequency modes are set to zero. With this truncation, one Fourier layer of the FNO is given by
\begin{equation*}
    \hat{v}
    \;\mapsto\;
    \hat{\mathcal{F}}^{-1}
    \left(
        \boldsymbol{T}_{\text{low}} \cdot \hat{\mathcal{F}}(\hat{v})
    \right),
\end{equation*}
where $\boldsymbol{T}_{\text{low}}$ denotes the restriction of $\boldsymbol{T}$ to the retained low-frequency Fourier modes. By stacking multiple such Fourier layers and interleaving them with pointwise nonlinearities, the Fourier Neural Operator is able to approximate nonlinear operators with global receptive fields while remaining computationally efficient.

\paragraph{Experimental configuration}
In our experiments, the FNO uses $\FNOlayers$ Fourier layers with channel width $\FNOwidth$, retaining $\FNOmodes$ low-frequency modes per spatial dimension in each spectral convolution layer.

\subsection{Convolutional neural operator (CNO)}
\label{app:cno}

The Convolutional Neural Operator (CNO) \citep{raonic2023convolutional} is designed to approximate operators
\(
\mathbf{G} : \mathcal{B}_w(D,\mathbb{R}^{d_X}) \to \mathcal{B}_w(D,\mathbb{R}^{d_Y})
\)
between band-limited function spaces. By restricting both the input and output to band-limited spaces, the CNO mitigates aliasing errors and enables stable learning across different spatial resolutions.

The CNO represents the target operator as a composition of lifting, convolutional operator layers, and projection, expressed as
\begin{equation*}
    \mathcal{G} : u \mapsto P(u) = v_0
    \mapsto v_1 \mapsto \cdots \mapsto v_L
    \mapsto Q(v_L) = \bar{u},
\end{equation*}
where $P$ is a lifting operator that maps the input function $u$ to a higher-dimensional feature representation, $Q$ is a projection operator to the target space, and the intermediate states $\{v_\ell\}_{\ell=0}^L$ are defined recursively by
\begin{equation*}
    v_{\ell+1}
    =
    \mathcal{P}_{\ell}
    \circ
    \Sigma_{\ell}
    \circ
    \mathcal{K}_{\ell}(v_{\ell}),
    \qquad 1 \le \ell \le L-1.
\end{equation*}
Here, $\mathcal{K}_{\ell}$ denotes a convolutional operator, $\Sigma_{\ell}$ a pointwise nonlinearity, and $\mathcal{P}_{\ell}$ a resolution-changing operator.

Architecturally, the CNO adopts a modified operator U-Net structure composed of an encoder that expands the feature width and contracts the spatial height, enabling multiscale feature extraction, and a decoder that reconstructs high-resolution outputs by successively upsampling the feature maps while reducing the channel dimension. To compensate for the potential loss of high-frequency information caused by upsampling, feature maps from earlier encoder stages are concatenated with decoder features, ensuring the preservation of fine-scale details.

\paragraph{Experimental configuration}
We use a $\CNOlevels$-level encoder--decoder CNO with base channel width $\CNOchannels$ and $\CNOrblocks$ residual (R-)blocks at each resolution.

\subsection{U-shaped neural operator (UNO)}
\label{app:uno}

UNO is a Fourier-based neural operator architecture that extends the Fourier Neural Operator (FNO) with a U-Net-style encoder-decoder structure and skip connections. UNO progressively contracts and expands the spatial domain of intermediate functions, enabling deeper and more memory-efficient operator learning while preserving global interactions through spectral integral operators. Its core operations are performed in the Fourier domain, making it particularly effective for learning global solution operators of PDEs.

In contrast, CNO replaces Fourier integral operators with purely convolutional operators defined in physical space. CNOs rely on hierarchical convolution, interpolation, and local kernel operations to achieve resolution-invariant operator learning. Unlike UNO, CNO does not explicitly use spectral transforms or global Fourier modes; instead, it emphasizes locality, translation equivariance, and scalability through convolutional architectures similar to CNNs.

Therefore, the key distinction is that UNO learns operators primarily through \emph{global spectral interactions} using Fourier operators and domain contraction/expansion, whereas CNO learns operators through \emph{local multiscale convolutions} directly in physical space.

\paragraph{Experimental configuration}
We use an $\UNOdepth$-stage UNO with base channel width $\UNOwidth$ in the lifting layer and scale the width across encoder/decoder stages accordingly.

\subsection{Deep operator network (DeepONet / PI-DeepONet)}
\label{app:deeponet}

Deep Operator Networks (DeepONets) \citep{lu2021learning} learn nonlinear operators by combining two subnetworks: a branch net that encodes the input function and a trunk net that encodes the query location. For an input coefficient field $\kappa$ and a spatial point $y\in\Omega$, DeepONet represents the solution operator in the form
\begin{equation}
u_\theta(y;\kappa)
=
\sum_{j=1}^{p} b_{\theta,j}(\kappa)\, t_{\theta,j}(y),
\label{eq:deeponet_separable}
\end{equation}
where $b_{\theta}(\kappa)\in\mathbb{R}^{p}$ is produced by the branch net and $t_{\theta}(y)\in\mathbb{R}^{p}$ is produced by the trunk net.
Evaluating~\eqref{eq:deeponet_separable} over all grid points yields the full predicted field.

Physics-informed variants of DeepONet (often referred to as PI-DeepONet or PI-DON) incorporate the governing PDE as an additional training signal. In our experiments, we adopt a PI-DeepONet-style training of DeepONet by using the common PI objective in Section~\ref{app:baseline_common}, where the physics loss is computed by automatic differentiation with respect to the trunk coordinates.

\paragraph{Experimental configuration}
For the plain DeepONet baseline, we set the latent dimension to $p=\DONp$. Both the branch and trunk nets are MLPs of depth $\DONbranchDepth$ and $\DONtrunkDepth$, respectively, with hidden width $\DONwidth$.

For the PI-DeepONet baseline, we use branch layers $(d_b,100,100,100,100)$ and trunk layers $(2,100,100,100,100)$, where $d_b=C_{\mathrm{in}}s^2$ is the flattened branch input dimension.

\subsection{Transolver}
\label{app:transolver}

Transolver \citep{Transolver} is an operator learning architecture built on physics-attention, which compresses dense attention over grid tokens into a two-stage slicing--token mechanism. Let $x\in\mathbb{R}^{N\times d}$ denote the per-point features on a grid ($N=s^2$). Transolver first assigns each grid point to one of $S$ latent slices by
\begin{equation*}
A=\mathrm{softmax}\!\left(\frac{W_s x}{\tau}\right)\in\mathbb{R}^{N\times S},
\end{equation*}
then forms slice tokens by normalized pooling,
\begin{equation*}
T=\frac{A^\top x}{A^\top \mathbf{1}}\in\mathbb{R}^{S\times d}.
\end{equation*}
Multi-head self-attention is applied among the $S$ tokens (cost $\mathcal{O}(S^2)$), and the updated tokens are projected back to the grid by deslicing, $\tilde{x}=A T'$. This reduces the quadratic cost in $N$ to $\mathcal{O}(NS+S^2)$ while retaining global information flow, which is beneficial for elliptic problems where long-range dependencies are intrinsic.

\paragraph{Experimental configuration}
We use feature dimension $d=\TSdim$, number of slices $S=\TSslices$, number of attention heads $\TSheads$, and temperature $\tau=\TStau$ for the slicing softmax.

\subsection{Physics-informed Neural Operator (PINO)}
\label{app:pino}

The Physics-Informed Neural Operator (PINO) combines an operator-learning backbone with a PDE residual during training. In our implementation, the PINO baseline uses an FNO-style 2D architecture on the regular $s\times s$ grid. Concretely, if $v_\ell$ denotes the hidden feature map at layer $\ell$, one PINO block takes the form
\[
v_{\ell+1}(x)
=
\sigma\!\left(\mathcal{K}_\ell(v_\ell)(x)+W_\ell v_\ell(x)\right),
\]
where $\mathcal{K}_\ell$ is a Fourier spectral convolution retaining only low-frequency modes and $W_\ell$ is a learnable pointwise linear map. Stacking such blocks yields a neural operator with a global receptive field and efficient spectral mixing.

For physics-informed training, PINO uses the same PI objective from Section~\ref{app:baseline_common}.

\paragraph{Experimental configuration}
We use four spectral blocks, hidden width $64$, retain $12$ Fourier modes in each spatial dimension at every block, use a pointwise projection dimension of $128$, GELU activations, and no domain padding.

\subsection{Finite-Element Operator Network (FEONet)}

FEONet is a neural operator model grounded in the finite element method (FEM), using nodal basis functions $\{\phi_i\}$ to approximate solutions as a linear combination
\begin{equation*}
    u = \sum_{i=1}^{N_H} \alpha_i \phi_i
\end{equation*}
where the index $i$ denotes the nodes of a triangulation of the physical domain and $N_H$ the number of (coarse) nodes. The basic idea of FEONet is to predict the FEM-coefficients $\{\alpha_i\}$. For the Darcy equation \ref{darcy_eq}, the training procedure of FEONet is based on the weak formulation, which reads: Find $u \in V$ such that
\[
\int_\Omega \kappa \nabla u \cdot \nabla v \, dx = \int_\Omega f v \, dx \quad \forall v \in V.
\]
Using the FEM ansatz $u = \sum_j \alpha_j \phi_j$ and testing with $\phi_i$, one obtains the linear system $A \alpha = b$ with the stiffness matrix $A = (A_{ij})$ and the load vector $f = (f_i)$ defined as 
\[
A_{ij} = \int_\Omega \kappa \nabla \phi_j \cdot \nabla \phi_i \, dx, \quad b_i = \int_\Omega f \phi_i \, dx.
\]
for $i,j=1,....,N_H$. FEONet is then trained by sampling permeability fields $\kappa_i$ and computing corresponding stiffness and load vector samples $\{(A_i(\kappa_i), b_i(\kappa_i))\}_{i=1}^N$ and then using the weak-form residual loss over samples with index $i$:
\[
\mathcal{L} = \frac{1}{N} \sum_{i=1}^N \| A_i \alpha_i - b_i \|.
\]
Note that, by using this loss function, we do not need the ground-truth FEM-coefficients $\alpha_i$ to train our model. This way, FEONet can be interpreted as a \textbf{hybrid} Operator Learning model that combines neural networks with the conventional FEM-method, which is \textbf{data-free} and \textbf{fully physics-informed}.

\paragraph{Experimental configuration}
In our implementation of FEONet, we use a CNN-based feature extractor for the input $\kappa$, followed by a Multi-Layer Perceptron to predict the coefficients $\alpha_i$ based on these features. For comparability with LOD-MSNO, the input of the model is the permeability field $\kappa$ on a grid with the fine resolution $h$ and the output is the coefficients $\alpha_i$ for the coarse-scale-basis functions $\phi_i$, $i=1,...,N_H$.

\section{Datasets}
\label{app:datasets}

\paragraph{Coefficient families}
We provide more detailed information about the used datasets here:
\begin{itemize}
\item \textbf{Quantile and filtered quantile coefficients.}
We first sample a smooth Gaussian random field $g$ on the fine nodes using the squared-exponential covariance
\begin{equation*}
\mathbb{E}\bigl[g(x)g(y)\bigr]
=
\sigma_{\mathrm{SE}}^2
\exp\!\left(-\frac{\|x-y\|_2^2}{2\ell_{\mathrm{SE}}^2}\right),
\end{equation*}
where $\ell_{\mathrm{SE}}=\texttt{se\_length}$ and $\sigma_{\mathrm{SE}}=\texttt{se\_sigma}$. The field is exponentiated, $a=\exp(g)$, and empirical quantile bins of $a$ are mapped to a randomly shuffled set of prescribed permeability values
\[
\{0.01,4,8,12,66,100\}
\]
by default. The \texttt{filtered\_quantile} variant additionally blends the nodal values toward the background value $4$ near the boundary using a smooth boundary weight. This is the default coefficient type in the generation script.

\item \textbf{Lognormal Mat\'ern coefficients.}
For \texttt{coeff\_type=lognormal}, we sample a Mat\'ern Gaussian field $Z$ on the fine nodes and define $\kappa=\exp(Z)$. The covariance is
\begin{equation*}
\mathbb{E}\bigl[Z(x)Z(x')\bigr]
=
\frac{\sigma^2}{2^{\nu-1}\Gamma(\nu)}
\left(
\frac{\sqrt{2\nu}}{\ell}\|x-x'\|_2
\right)^{\nu}
K_{\nu}\!\left(
\frac{\sqrt{2\nu}}{\ell}\|x-x'\|_2
\right),
\end{equation*}
where $K_{\nu}$ denotes the modified Bessel function of the second kind. The default Mat\'ern parameters are $\nu=0.5$, correlation length $\ell=0.3$, variance parameter $\sigma=2.0$, and jitter $10^{-8}$, with the same implementation also allowing the lower-contrast lognormal setting by changing these arguments.

\item \textbf{Fine checkerboard coefficients.}
For \texttt{coeff\_type=fine\_checkerboard}, values are sampled independently from the prescribed permeability levels on the fine $h$-scale grid cells. Nodal coefficient values are obtained by indexing into the corresponding fine cell.

\end{itemize}

\paragraph{Per-sample LOD assembly}
For each generated coefficient realization, the script performs the following steps:
\begin{enumerate}
\item \textbf{Fine stiffness assembly.}
The elementwise coefficient is used to assemble a discontinuous Galerkin element stiffness matrix, which is then mapped to the continuous fine space to obtain the fine stiffness matrix $\boldsymbol{A}_h$. The load corresponds to the constant source $f\equiv 1$ and is assembled using the fine mass matrix.

\item \textbf{Localized correctors.}
For each coarse element $T$, the localized corrector is computed on the precomputed fine patch degrees of freedom. In matrix form, the implementation solves the constrained local corrector problem using the local stiffness block, the local right-hand side induced by the three coarse basis functions on $T$, and the patch interpolation constraint $\boldsymbol{I}_{H,T}$. The resulting local corrector matrix has shape $N_{\ell,h}(T)\times 3$ and is stored in the flattened array \texttt{corr\_values}, with \texttt{corr\_ptr} marking the offsets of each patch.

\item \textbf{Global multiscale operator.}
The local correctors are inserted into global fine-grid columns and assembled into the localized correction matrix $\boldsymbol{C}^{\ell}$. The multiscale reconstruction matrix is then
\begin{equation*}
\boldsymbol{G}=\boldsymbol{P}_1-\boldsymbol{C}^{\ell}.
\end{equation*}
After restricting to interior coarse degrees of freedom, the LOD stiffness matrix and load vector are formed as
\begin{equation*}
\boldsymbol{A}_{\mathrm{LOD}}
=
\boldsymbol{G}_0^{\top}\boldsymbol{A}_h\boldsymbol{G}_0,
\qquad
\boldsymbol{f}_{\mathrm{LOD}}
=
\boldsymbol{G}_0^{\top}\boldsymbol{b}_h.
\end{equation*}

\item \textbf{Coarse solve and fine reconstructions.}
The interior LOD system is solved for $\boldsymbol{u}_{\mathrm{LOD}}$, the full coarse vector is padded with zero boundary values, and the fine-grid multiscale solution is reconstructed as
\begin{equation*}
\boldsymbol{u}_{\mathrm{ms}}=\boldsymbol{G}\boldsymbol{u}_{H}.
\end{equation*}
For diagnostic and evaluation purposes, the fine FEM reference solution $\boldsymbol{u}_h$ is also computed by solving the fine interior system with homogeneous Dirichlet boundary conditions.
\end{enumerate}

\section{Implementation details}

\subsection{Neural architecture for LOD-MSNO (Coeff)}
\label{app:arch_lodmimetic}

\paragraph{Input preprocessing}

We form a multi-channel tensor $X\in\mathbb{R}^{B\times C_{\mathrm{in}}\times n_h\times n_h}$ from the raw field $a_h$.
Denote by $b$ the \emph{base channel}:
\begin{equation}
    b(i, j) = \log(\max\{a_h(i,j),\varepsilon\})
\end{equation}
where $\varepsilon>0$ is a small constant to avoid numerical issues. The log transform compresses the dynamic range of high-contrast permeability fields. This is consistent with the common modeling practice in Darcy-flow settings \citep{MR3283907}.

\paragraph{Gradient-magnitude channel (\texttt{add\_grad}).}
To expose sharp interfaces/discontinuities in $a_h$, we append a gradient-magnitude channel computed by forward differences
(with zero padding at the boundary):
\begin{align}
g_x(i,j) &= 
\begin{cases}
b(i,j+1)-b(i,j), & j<n_h-1,\\
0, & j=n_h-1,
\end{cases}
\\
g_y(i,j) &=
\begin{cases}
b(i+1,j)-b(i,j), & i<n_h-1,\\
0, & i=n_h-1,
\end{cases}
\\
g(i,j) &= \sqrt{g_x(i,j)^2 + g_y(i,j)^2 + \delta},
\label{eq:preproc_gradmag}
\end{align}
with $\delta>0$ small.

\paragraph{Coordinate channels (\texttt{add\_coords}).}
Convolutional networks are translation-equivariant, while boundary-value problems on bounded domains are position-dependent.
We therefore append explicit coordinates (CoordConv-style positional encoding \citep{10.5555/3327546.3327630}):
\begin{equation}
x(j)=-1+\frac{2j}{n_h-1},
\qquad
y(i)=-1+\frac{2i}{n_h-1},
\label{eq:preproc_coords}
\end{equation}

\paragraph{Channel count}
The number of input channels is
\begin{equation}
C_{\mathrm{in}} \;=\; 1 \;+\; \mathbb{I}[\texttt{add\_grad}] \;+\; 2\,\mathbb{I}[\texttt{add\_coords}],
\label{eq:Cin}
\end{equation}
and the final preprocessed tensor is formed by concatenation along the channel dimension.

Let $X\in\mathbb{R}^{B\times C_{\mathrm{in}}\times n_h\times n_h}$ be the preprocessed input.

\paragraph{Channel-wise input normalization.}
We compute per-channel statistics $(\mu_c,\sigma_c)$ on the training set after preprocessing:
\begin{equation}
\mu_c \;=\; \mathbb{E}[X_{c}],\qquad
\sigma_c^2 \;=\; \mathbb{E}[X_{c}^2]-\mu_c^2,
\label{eq:x_stats}
\end{equation}
where expectations are taken over all training samples and spatial locations.
We then normalize as
\begin{equation}
\widetilde{X}_{b,c,i,j}
\;=\;
\frac{X_{b,c,i,j}-\mu_c}{\sigma_c+\epsilon}.
\label{eq:x_norm}
\end{equation}
The values $\mu,\sigma$ are stored as non-trainable buffers and reused at test time.

\paragraph{Per-component output normalization.}
Let $u_{\mathrm{int}}\in\mathbb{R}^{D}$ denote the target interior coefficient vector.
We compute per-component statistics $(\mu^{(u)}_k,\sigma^{(u)}_k)$ on the training set and normalize:
\begin{equation}
\widetilde{u}_k
\;=\;
\frac{u_{\mathrm{int},k}-\mu^{(u)}_k}{\sigma^{(u)}_k+\epsilon}.
\label{eq:u_norm}
\end{equation}
The core network is trained to predict $\widehat{\widetilde{u}}$; inference returns the decoded coefficients
$\widehat{u}_{\mathrm{int}}$ via the inverse transform of \eqref{eq:u_norm}.

\paragraph{Building block for CNN-based model: GNAct.}
We use the block
\begin{equation}
\mathrm{GNAct}(z)
\;=\;
\mathrm{GELU}\!\left(\mathrm{GN}\!\left(\mathrm{Conv}_{k\times k,s}(z)\right)\right),
\label{eq:gnact}
\end{equation}
where $\mathrm{GN}$ denotes Group Normalization \citep{Wu_2018_ECCV}.
For GroupNorm, we use $G=\min(8,C)$ groups for a $C$-channel tensor.

\paragraph{Residual block (ResBlock2D).}
At a fixed resolution and channel width $C$, we use a ResNet-style block:
\begin{equation}
    \mathrm{ResBlock}(x) = \mathrm{GELU}\left(x + \mathrm{GN} \left( \mathrm{Conv}_{3 \times 3, 1}\left( \mathrm{GNAct}(x) \right) \right) \right)
    \label{eq:resblock}
\end{equation}

\paragraph{Hyperparameters used.}
We use $\texttt{base}=32$, GroupNorm with $G=\min(8,C)$ groups,
$\texttt{add\_grad}=\texttt{True}$, and $\texttt{add\_coords}=\texttt{True}$.

\subsection{Neural architecture for predicting the basis correction}
\label{app:basis_architecture}

\paragraph{Patch-wise supervised learning target}

Each training example corresponds to one coefficient realization and one coarse patch. For a fixed patch $T\in\mathcal{T}_H$, the input is the coefficient field restricted to the localized region $\mathcal{N}^{\ell}(T)$ and stored on a rectangular patch grid. Since boundary and interior patches can have different physical sizes and different active fine degrees of freedom, each patch is embedded into a common rectangular patch canvas using padding.
Thus, all patches have the same input shape and can be processed by one convolutional network. 

The target consists of the local correction coefficients
\begin{equation}
\boldsymbol{w}_{T,j}(\omega)
=
\bigl(w_{T,j,1}(\omega),\ldots,w_{T,j,N_{\ell,h}(T)}(\omega)\bigr)^{\top},
\qquad j=1,\ldots,N_H,
\end{equation}
restricted in practice to the coarse basis functions whose support intersects $T$. Since a two-dimensional triangular coarse element has three local coarse vertices, the implementation stores the local corrector target as an array with three output components per fine patch degree of freedom. Thus, for a patch with $N_{\ell,h}(T)$ active fine degrees of freedom, the supervised target has shape
\begin{equation}
Y_T(\omega)\in\mathbb{R}^{N_{\ell,h}(T)\times 3}.
\end{equation}
The basis dataset therefore consists of patch-sample pairs
\begin{equation}
\mathcal{D}_{\mathrm{basis}}
=
\left\{
\left(\kappa_i|_{\mathcal{N}^{\ell}(T)},Y_T(\omega_i)\right)
:\; i=1,\ldots,M,\; T\in\mathcal{T}_H
\right\}.
\end{equation}
In the implementation, this dataset is flattened over samples and patches, so that one training epoch iterates over all pairs $(i,T)$.

\paragraph{PatchCorrectionNet architecture}

The core of \texttt{PatchCorrectionNet} is a U-Net-style convolutional encoder--decoder with skip connections. It uses the same basic Conv--GroupNorm--GELU and residual blocks as the coefficient architecture, but preserves the patch-grid structure and returns a dense three-channel patch image. Starting from
\begin{equation}
X_T\in\mathbb{R}^{B\times 3\times H_p\times W_p},
\end{equation}
the network applies the following stages:
\begin{itemize}
\item \textbf{Stem:} a non-downsampling \texttt{DownStage} mapping
$3\rightarrow c_1$ channels, followed by residual blocks.
\item \textbf{Encoder:} three stride-two downsampling stages with channel widths
\begin{align*}
 c_1&=\texttt{base}, &
 c_2&=\min(2c_1,\texttt{max}), \\
 c_3&=\min(4c_1,\texttt{max}), &
 c_4&=\min(8c_1,\texttt{max}).
\end{align*}
Each stage consists of a Conv--GroupNorm--GELU block followed by \texttt{blocks\_per\_stage} residual blocks.
\item \textbf{Bottleneck:} \texttt{bottleneck\_blocks} residual blocks at width $c_4$.
\item \textbf{Decoder:} three upsampling stages using bilinear interpolation, concatenation with the corresponding encoder skip feature, a Conv--GroupNorm--GELU projection, and residual blocks.
\item \textbf{Head:} a Conv--GroupNorm--GELU block, one residual block, and a final $1\times1$ convolution producing
\begin{equation}
\widehat{Y}_T^{\mathrm{grid}}(\omega)
\in
\mathbb{R}^{B\times 3\times H_p\times W_p}.
\end{equation}
\end{itemize}
The three output channels correspond to the correction coefficients associated with the three local coarse basis functions on the patch. The architecture is fully convolutional on the patch canvas, and the use of skip connections is important because local correctors are spatially localized and may contain sharp features near coefficient jumps or patch boundaries.

\paragraph{Patch readout and constraint projection}

The network output is a dense patch image, whereas the correction target is stored only on active fine degrees of freedom. Let
\begin{equation}
\mathcal{C}_T=\{(x_m,y_m)\}_{m=1}^{N_{\ell,h}(T)}
\end{equation}
denote the list of local grid coordinates corresponding to these active fine degrees of freedom. The dense prediction is read out as
\begin{equation}
\widehat{Y}_T(\omega)
=
\left(
\widehat{Y}_T^{\mathrm{grid}}(:,y_m,x_m)
\right)_{m=1}^{N_{\ell,h}(T)}
\in\mathbb{R}^{N_{\ell,h}(T)\times 3}.
\end{equation}
Because the LOD correctors must lie in the local fine-scale space $W^{\ell}(T)=\ker(\mathcal{I}_H)$, the implementation applies a hard algebraic projection before computing the loss. With the precomputed patch matrices denoted by $\boldsymbol{P}_{T}^{\mathrm{left}}$ and $\boldsymbol{P}_{T}^{\mathrm{right}}$, the projected prediction is
\begin{equation}
\widehat{Y}_{T}^{\mathrm{proj}}
=
\widehat{Y}_T
-
\boldsymbol{P}_{T}^{\mathrm{left}}
\boldsymbol{P}_{T}^{\mathrm{right}}
\widehat{Y}_T,
\end{equation}
and the same projection is applied to the target $Y_T$. This projection enforces compatibility with the kernel constraint induced by $\mathcal{I}_H$ and prevents the network from being penalized for components outside the admissible local detail space.

\paragraph{Hyperparameters used}
For \texttt{PatchCorrectionNet}, we use
\begin{align*}
\texttt{base\_channels} &= 48, &
\texttt{max\_channels} &= 256, &
\texttt{blocks\_per\_stage} &= 2, \\
\texttt{bottleneck\_blocks} &= 4, &
\texttt{dropout} &= 0.05.
\end{align*}
The output dimension is fixed to three channels, corresponding to the three local coarse basis functions associated with a triangular coarse element. During validation, predictions from all patches are assembled into the global correction matrix $\widehat{\boldsymbol{Q}}_h^{\ell}$ and combined with the coarse coefficients to reconstruct the fine-scale LOD solution.

\end{document}